\documentclass[12pt]{amsart}

\usepackage[utf8]{inputenc}
\usepackage[USenglish]{babel}
\usepackage{amsmath,amsthm,amssymb,amsfonts}
\usepackage{mathtools}
\usepackage{mathrsfs}
\usepackage{bbm}

\usepackage{indentfirst}
\usepackage{enumitem}

\usepackage{float}

\usepackage[breaklinks=true, bookmarksopenlevel=1, bookmarksdepth=2]{hyperref}

\allowdisplaybreaks

\newtheorem{thm}{Theorem}[section]
\newtheorem{cor}[thm]{Corollary}
\newtheorem{lem}[thm]{Lemma}

\newtheorem{conj}[thm]{Conjecture}

\theoremstyle{plain} % just in case the style had changed
\newcommand{\thistheoremname}{}
\newtheorem*{genericthm}{\thistheoremname}
\newenvironment{nthm}[1]
  {\renewcommand{\thistheoremname}{#1}%
   \begin{genericthm}}
  {\end{genericthm}}

\theoremstyle{definition}

\theoremstyle{remark}
\newtheorem{rem}[thm]{Remark}
\newtheorem*{xrem}{Remark}
\newtheorem*{prob}{Problem}
\newtheorem*{ntt}{Notation}

\numberwithin{equation}{section}

\newcommand{\N}{\mathbb{N}}      % N = Naturals
\newcommand{\Z}{\mathbb{Z}}      % Z = Integers
\newcommand{\R}{\mathbb{R}}      % R = Reals
\newcommand{\eps}{\varepsilon}   % epsilon

\newcommand{\negphantom}[1]{\settowidth{\dimen0}{$\displaystyle #1$}\hspace*{-\dimen0}}

\newcommand{\E}{\mathbb{E}}   % Expectation
\newcommand{\as}{\textnormal{a.s.}}
\newcommand{\A}{\mathscr{A}}

\renewcommand{\pmod}[1]{         % (mod #1)
  ~(\mathrm{mod}~#1)}

\restylefloat{table}

\setlist[itemize]{leftmargin=*}
\setlist[enumerate]{leftmargin=*}

\begin{document}

\title{Representation functions with prescribed rates of growth}%
\author{Christian T\'afula}%
\address{Instituto de Matem\'atica, Estat\'istica\\
e Ci\^encia da Computa\c{c}\~ao\\
Universidade de S\~ao Paulo\\
Rua do Mat\~ao, 1010\\
S\~ao Paulo, SP 05508-090\\
Brazil}
\curraddr{}
\email{tafula@ime.usp.br}
\thanks{}


% ----------------------------------------------------------------
 \begin{abstract}
  Fix an integer $h \geq 2$, and let $b_1, \ldots, b_h$ be (not necessarily distinct) positive integers with $\gcd(b_1, \ldots, b_h) = 1$. For any subset $A \subseteq \mathbb{N}$, let $r_A(n)$ denote the number of solutions $(k_1, \ldots, k_h) \in A^h$ to the equation
  \[ b_1 k_1 + \cdots + b_h k_h = n. \]
  Given a function $F$ satisfying $F(n) \leq r_{\mathbb{N}}(n)$, we ask: when does there exist a set $A \subseteq \mathbb{N}$ such that $r_A(n) \sim F(n)$? We prove that this is always possible when $F$ is regularly varying and satisfies $\lim_{n\to\infty} F(n)/\log n = \infty$.

  If one only requires $r_A(n) \asymp F(n)$, much weaker regularity conditions suffice: we show such a set $A$ exists for every increasing function $F$ satisfying $F(2x) \ll F(x)$ and $\log x \ll F(x) \ll x^{h-1}$.

  Finally, we give a probabilistic heuristic supporting the following: if $A \subseteq \mathbb{N}$ satisfies $\limsup_{n\to\infty} r_A(n)/\log n < 1$, then $r_A(n) = 0$ for infinitely many $n$.
 \end{abstract}

\maketitle
% ----------------------------------------------------------------

%%%%%%%%%%%%%%%%%%%%%%%%%%%%%%%%%%%%%%%%%%%%%%%%%%%%%%%%%
\section{Introduction}
 Fix an integer $h \geq 2$ and positive integers $b_1, \ldots, b_h \in \Z_{\geq 1}$ (not necessarily distinct) with $\gcd(b_1, \ldots, b_h) = 1$. For a subset $A\subseteq \Z_{\geq 0}$ and $1\leq \ell \leq h$, define for each $n\geq 0$ the function
 \begin{equation}
  r_{A,\ell}(n) = r^{(b_1,\ldots,b_h)}_{A,\ell} := \#\{(k_1, \ldots, k_{\ell}) \in A^{\ell} ~|~ b_1 k_1 + \cdots + b_{\ell} k_{\ell} = n \}; \label{mainEq}
 \end{equation}
 i.e., the number of solutions to $b_1 k_1 + \cdots + b_{\ell} k_{\ell} = n$ in $A$. It is a classical result in additive combinatorics %\footnote{Formerly known as additive number theory.}
 that there are 
 \begin{equation}
   \bigg(1+ O\bigg(\frac{1}{n}\bigg)\bigg) \frac{1}{(h-1)!} \frac{n^{h-1}}{b_1\cdots b_h} \label{tuppbd}
 \end{equation}
 nonnegative integer solutions to the equation $b_1k_1+\cdots + b_{h}k_h = n$.
 
 In 1990, Erd\H{o}s--Tetali \cite{erdtet90} showed that there exists $A\subseteq \Z_{\geq 0}$ for which $r_{A,h}(n) \asymp \log n$, provided $b_1 = \ldots = b_h = 1$. In 2000, Vu \cite{vvu00mc} showed this for general $b_1,\ldots,b_h$. In this paper, we will address the following (cf. Nathanson \cite[Problem 1]{nat05}):
 
 \begin{nthm}{Question}
  For which functions $1\ll F(x) \ll x^{h-1}$ can we find a subset $A\subseteq \Z_{\geq 0}$ satisfying $r_{A,h}(n) \sim F(n)$? Or $r_{A,h}(n) \asymp F(n)$?
 \end{nthm}
 
 \begin{xrem}[Prescribed representations in $\Z$]
  Suppose we are allowed to use negative integers as well as non-negative. Nathanson \cite{nat05} showed that if $b_1 = \ldots = b_h = 1$, then for every $F:\Z\to \Z_{\geq 0}$ with finitely many zeros there exists $A\subseteq \Z$ with $r'_{A,h}(n) = F(n)$, where $r'$ counts the number of \emph{distinct} solutions (i.e., ignoring permutations) to $b_1k_1 +\cdots + b_{h}k_h=n$. Fang \cite{fang24} showed that sets of $b_1,\ldots,b_h$ with distinct subset sums also enjoy this property.
 \end{xrem}
 
 In order to attack Nathanson's question with analytic methods, we will assume certain regularity conditions on $F$, and that $F(x) \gg \log x$. In fact, Erd\H{o}s \cite[p. 132]{erdo56} inquired whether, if $b_1=b_2=1$ and $r_{A,2}(n)>0$ for large $n$, it is always the case that
 \begin{equation}
  \limsup_{n\to \infty} \frac{r_{A,2}(n)}{\log n} >0. \label{erdconj}
 \end{equation}
 This conjecture\footnote{This is a strong version of the Erd\H{o}s--Tur\'an conjecture for additive bases: Does $A+A=\Z_{\geq 0}$ necessarily imply that $r_{A,2}(n)$ is unbounded? More generally, if $|A\cap[1,x]| \gg x^{1/2}$, then is $r_{A,2}(n)$ necessarily unbounded?} naturally extends to general $r_{A,h}$ as defined in \eqref{mainEq}, and a heuristic in its favor is the subject of Theorem \ref{MT3}.
 
 Our existence proofs are probabilistic: we consider a random set $\A\subseteq \Z_{\geq 0}$ obtained by including each integer $n$ independently with probability $p(n)$, chosen so that $\E(r_{\A,h}(n))$ matches the target growth $F(n)$; we then show the desired bounds hold almost surely, hence for some deterministic set $A$. (See Section \ref{probset} for the precise choice of $p$ and the basic expectation computations.)
 
\subsection{Asymptotics}
 In the asymptotic case, we will work with \emph{regularly varying} functions $F$; that is, $F$ is a positive measurable function, and for every $\lambda>0$, $\lim\limits_{x\to \infty} F(\lambda x)/F(x)$ exists. Every regularly varying function is of the form
 \[ F(x) = x^{\kappa} \phi(x) \]
 for some real $\kappa \in \R$ and some \emph{slowly varying function} $\phi(x)$, meaning that $\phi$ satisfies $\lim\limits_{x\to\infty} \phi(\lambda x)/\phi(x) = 1$ for every real $\lambda >0$.\footnote{e.g. $\log x$, $e^{c\sqrt{\log x}}$, but not $(\log x)^{\sin x}$.} See Bingham--Goldie--Teugels \cite{bingham89} for an excellent treatise on the subject.
 
 Our first theorem deals with $F$ such that $F(x)/\log x \to \infty$ as $x\to \infty$, where Kim--Vu's inequality \cite{kimvvu00} paired with an inequality by Vu \cite[Theorem 1.4]{vvu00mc} allows us to obtain asymptotics.

 \begin{thm}\label{MT1}
  Let $h\geq 2$ be a given integer. Let $F$ be a regularly varying function for which
  \[ \frac{F(x)}{\log x} \xrightarrow{x\to\infty} \infty\quad\text{and}\quad F(x) \leq (1+o(1)) \frac{1}{(h-1)!} \frac{x^{h-1}}{b_1\cdots b_{h}}. \]
  Then, there exists $A\subseteq \Z_{\geq 0}$ such that $r_{A,h}(n) \sim F(n)$.
 \end{thm}
 
 Note that by \eqref{tuppbd}, the upper bound covers as wide a range as possible. The proof also implies that we can take $A$ satisfying $|A\cap [1,x]| \sim C\, (x F(x))^{1/h}$, where, if $F(x) = x^{\kappa} \phi(x)$,
 \[ C = \frac{h}{1+\kappa} \frac{\Gamma(1+\kappa)^{\frac{1}{h}}}{\Gamma(\tfrac{1+\kappa}{h})} (b_1\cdots b_{h})^{(1+\kappa)/h^2}. \]
 When $F(x) = x^{\kappa}$, we obtain power savings from Kim--Vu's inequality.
 
 \begin{cor}\label{MT12}
  Let $h\geq 2$ be a given integer, and $0<\kappa \leq h-1$ a real number. Then, for any $C\in \R_{>0}$, there exists $A\subseteq \Z_{\geq 0}$ such that $|A\cap [1,x]| = C x^{(1+\kappa)/h} + O(x^{(1+\kappa)/2h}\log x)$ and
  \[ r_{A,h}(n) = C^{h}\frac{(1+\kappa)^{h}}{h^{h}} \frac{\Gamma(\tfrac{1+\kappa}{h})^{h}}{\Gamma(1+\kappa)}\, \frac{n^{\kappa}}{(b_1\cdots b_{h})^{\frac{1+\kappa}{h}}} + O(E_{h,\kappa}), \]
  where $E_{2,\kappa} := n^{\frac{\kappa}{2}}(\log n)^2$, $E_{3,\kappa} := n^{\frac{\kappa}{2}+\max\{0,\frac{\kappa}{3}-\frac{1}{2}\}}(\log n)^3$, and for $h\geq 4$,
  \[ E_{h,\kappa} := \begin{cases}
   n^{\frac{\kappa}{2}}(\log n)^{h} &\text{ if } 0\leq \kappa \leq \frac{2}{h-2},\\
   n^{(1-\frac{1}{h})\kappa - \frac{1}{h}} &\text{ if } \frac{2}{h-2} < \kappa < h-2, \\
   n^{(1-\frac{1}{2h})\kappa-\frac{1}{2}}(\log n)^h &\text{ if } h-2\leq \kappa \leq h-1.
  \end{cases} \]
 \end{cor}
 
 \begin{xrem}
  An interesting subcase is the following: Writing $r_{A,h+1}(n) = r^{(1,\ldots,1)}_{A,h+1}(n)$ for the number of solutions $(k_1,\ldots,k_{h+1})\in A^{h+1}$ to $k_1+\cdots +k_{h+1} = n$, we obtain the existence of a set $A\subseteq \Z_{\geq 0}$ with $|A\cap[1,x]| = x^{1/h} + O(x^{1/2h}\log x)$ and 
  \[ r_{A,h+1}(n) = \Gamma(1+1/h)^h\, n^{1/h} + O\big(n^{1/2h}(\log n)^{h+1}\big). \]
 \end{xrem}
 
\subsection{Order of magnitude}
 Under significantly weaker regularity hypotheses, one can still show the existence of sets $A\subseteq \Z_{\geq 0}$ such that $r_{A,h}(n) \asymp F(n)$, though we no longer obtain asymptotics.
 
 \begin{thm}\label{MT21}
  Let $h\geq 2$ be a fixed integer, and let $F$ be a positive, increasing, locally integrable real function satisfying $F(2x)\ll F(x)$, and in the range
  \begin{equation*}
   \log x \ll F(x) \ll x^{h-1} 
  \end{equation*}
  Then, there exists $A\subseteq \Z_{\geq 0}$ such that $|A\cap [1,x]| \asymp (xF(x))^{1/h}$ and 
  $r_{A,h}(n) \asymp F(n)$.
 \end{thm}
 
 This range includes all regularly varying functions $x^{\kappa}\phi(x)$ with $\kappa>0$, since $x^{\kappa}\phi(x) \asymp F(x)$ for some increasing $F$. Moreover, it includes some functions of the form $x^{k(x)}$ where $k(x)$ varies, such as $F(x) = x^{2+\sin(\log\log x)}$. The assumption $F(2x)\ll F(x)$ prevents rapid fluctuations but allows for mild oscillation.
 
 Lastly, we prove a more technical, direct generalization of the main theorem in \cite{taf19}, with the weakest regularity assumptions our methods allow.
  
 \begin{thm}\label{MT2}
  Let $h\geq 2$ be a given integer, and let $\psi(x) \gg \log x$ be an increasing slowly varying function. If
  \begin{enumerate}[label=\textnormal{(\roman*)}]
   \item \textnormal{(Range)} $(x\, \psi(x))^{1/h} \ll f(x)\ll \min\{ (x\, \psi(x))^{1/(h-1)},\ x\}$,\smallskip
   
   \item \label{item2thm}\textnormal{(Regularity)} $\int_{1}^{x} \frac{f(t)}{t}\,\mathrm{d}t \asymp f(x)$;
  \end{enumerate}
  then there exists $A\subseteq \Z_{\geq 0}$ such that $|A\cap[1,x]|\asymp f(x)$ and $r_{\A,h}(n) \asymp \dfrac{f(n)^h}{n}$.
 \end{thm}
 
 Writing $f(x) := (x F(x))^{1/h}$, Theorem \ref{MT2} implies that for every positive, locally integrable $F$ in the range $\psi(x) \ll F(x) \ll x^{\frac{1}{h-1}} \psi(x)^{1+\frac{1}{h-1}}$ that satisfies
 \[ \frac{1}{x}\int_{1}^{x} F(t)\,\mathrm{d}t \asymp F(x) \]
 (this is a consequence of Lemma \ref{orpiCHAR}), there exists $A\subseteq \Z_{\geq 0}$ such that $r_{A,h}(n) \asymp F(n)$.
 
 \begin{prob}
  Can one extend the range of Theorem \ref{MT2} to $(x\log x)^{1/h} \ll f(x) \ll x$?
 \end{prob}

\subsection{What about \texorpdfstring{$F(x) \ll \log(x)$}{F(x) << log(x)}?}
 Again taking $f(x) := (xF(x))^{1/h}$, we will use the probabilistic space of random subsets $0\in\A\subseteq \Z_{\geq 0}$ generated by the measure $\Pr(n\in \A)= \min\{c\frac{f(n)}{n}, 1\}$ ($n\geq 1$), where $c>0$ is some real number, which is constructed in order to have, essentially,
 \[ |\A\cap [1,x]| \asymp f(x) \text{ with probability } 1,\footnotemark\quad\text{and } \E(r_{\A,h}(n)) \asymp c^{h}F(n).\]
 Theorems\footnotetext{Precisely, $\Pr(\A \subseteq \Z_{\geq 0} ~|~ \exists C_1,C_2 \in \R_{>0} : \forall x\geq 1,\ C_{1} f(x) \leq |\A\cap [1,x]| \leq C_2 f(x)) = 1$.} \ref{MT1}--\ref{MT2} are proven by showing that $r_{\A,h}$ concentrates around its mean; e.g., $r_{\A,h}(n) \asymp \E(r_{\A,h}(n))$ almost surely in Theorems \ref{MT21}, \ref{MT2}.\footnote{Since finite intersections of events of probability $1$ have probability $1$ (and hence are non-empty), there must exist $A\subseteq \Z_{\geq 0}$ satisfying both $|A\cap [1,x]| \asymp f(x)$ and $r_{A,h}(n)\asymp F(n)$.} The next theorem shows that if $cf(n)$ is too small (i.e., the space is constructed so that $\E(r_{\A,h}(n))$ is small), then not only will $r_{\A,h}$ not concentrate, but also have infinitely many zeros almost surely.

 \begin{thm}\label{MT3}
  Fix $0<\eps<1$. Define the random set $0\in \A\subseteq \Z_{\geq 0}$ by
  \[ \Pr(n\in\A) = \min\bigg\{c\frac{(n\log(n))^{1/h}}{n}, 1\bigg\} \qquad (n\geq 1), \]
  for $c = (1-\eps)^{1/h} (b_1\cdots b_h)^{1/h^2}/\Gamma(\frac{1}{h})$. Then $\E(r_{\A,h}(n)) \sim (1-\eps)\log n$, but
  \begin{equation}
   \Pr(r_{\A,h}(n) = 0 \textnormal{ infinitely often}) = 1. \label{pr0}
  \end{equation}
 \end{thm}
 
 This suggests a stronger version of \eqref{erdconj}:
 
 \begin{conj}\label{MC1}
  If $r_{A,h}(n) > 0$ for all large $n$, then
  \[ \limsup_{n\to \infty} \frac{r_{A,h}(n)}{\log n} \geq 1. \]
 \end{conj}
 
 Note that the existence of thin bases does not directly contradict Conjecture \ref{MC1} -- cf., for instance, the constructions described in Nathanson \cite{nath10}.\footnote{E.g. Raikov--St\"{o}hr's construction: In the case $h=2$, take $A := S_1\sqcup S_2$, where $S_1$ consists of only those non-negative integers which can be written as a sum of odd powers of $2$, and $S_2$ of even powers. One has that $|A\cap [1,x]|\ll x^{1/2}$, yet $A+A = \Z_{\geq 0}$. However, the numbers $n_1 := 2 = 10_2$, $n_2 := 10 = 1010_2$, $n_3 := 42 = 101010_2$, etc., can be shown to have at least $2^{k-1}$ ($> \sqrt{n_k/3}$) representations.}
 
 \begin{ntt}
  Throughout the paper, we use the common asymptotic notation $\sim$, $o$, $O$, $\gg$, $\ll$, $\asymp$, with subscripts indicating the dependence of implied constants on parameters (omitting dependencies on $h$, $f$, $F$, and $b_1,\ldots,b_h$). Given two sequences of random variables $(X_n)_n$ and $(Y_n)_n$, we write $X_n \stackrel{\as}{\sim} Y_n$ if $\Pr(\lim_{n\to\infty} X_n/Y_n = 1)=1$ (assuming $Y_n\neq 0$ for all sufficiently large $n$). We write $X_n \stackrel{\as}{\ll} Y_n$ if there exists an event $\Omega_0$ with $\Pr(\Omega_0)=1$ and a finite random variable $C(\omega)$ such that, for all $\omega\in\Omega_0$, one has $X_n(\omega)\leq C(\omega)Y_n(\omega)$ for all sufficiently large $n$.
 \end{ntt}

%%%%%%%%%%%%%%%%%%%%%%%%%%
%%%%%%%%%%%%%%%%%%%%%%%%%%  
%%%%%%%%%%%%%%%%%%%%%%%%%%
\section{Probabilistic setup}\label{probset}
 Let $f(x)$ be a positive, locally integrable real function such that $f(x)\ll x$ and satisfying
 \begin{equation}
  \int_{1}^{x} \frac{f(t)}{t}\,\mathrm{d}t \asymp f(x). \label{orpi}
 \end{equation}
 Consider the probability space over subsets $\A\subseteq \Z_{\geq 0}$ with $0\in\A$ and
 \begin{equation}
  \Pr(\mathbbm{1}_{\A}(n)=1) = \E(\mathbbm{1}_{\A}(n)) := \min\bigg\{c\, \frac{f(n)}{n},\ 1\bigg\} \qquad (\forall n\in\Z_{\geq 1}) \label{presp}
 \end{equation}
 for some constant $c>0$ to be chosen later, where the $\mathbbm{1}_{\A}(n)$'s are mutually independent boolean random variables.\footnote{cf. Chapter III of Halberstam--Roth \cite{halberstam83} for a construction of the product measure.} The purpose of this space is to have the counting function of $\A$ be of the same order of magnitude as $f$: by the strong law of large numbers, we have $|\A\cap [1,x]| \stackrel{\textnormal{a.s}}{\sim} c \sum_{n\leq x} \frac{\min\{f(n),\, n/c\}}{n} \asymp_c f(x)$.
 
 We work under stronger hypotheses to prove Theorems \ref{MT1}--\ref{MT21}, but functions satisfying \eqref{orpi} capture the ``minimal assumptions'' necessary to prescribe an order of growth using \eqref{presp}, and will be used in Theorem \ref{MT2}.
 
 \begin{lem}[Characterization]\label{orpiCHAR}
  A positive, locally integrable real function $f$ satisfies $\int_{1}^{x} \frac{f(t)}{t}\,\mathrm{d}t\asymp f(x)$ if and only if:
  \begin{enumerate}[label=\textnormal{(\roman*)}]
   \item For any $\lambda > 0$, we have $f(\lambda x) \asymp_{\lambda} f(x)$; and\medskip
   
   \item \label{itemii} There is $\vartheta = \vartheta_{f} >0$ for which the following holds: there exists $x_{0} \in \R_{>0}$ and $M>0$ such that, for every $y>x\geq x_0$, we have $\dfrac{f(x)}{x^{\vartheta}} \leq M \dfrac{f(y)}{y^{\vartheta}}$.
  \end{enumerate}
 \end{lem}
 \begin{proof}
  This is Corollary 2.6.2 of Bingham--Goldie--Teugels \cite{bingham89}, to which we give a short proof for the sake of completeness.\medskip
  
  \noindent
  $\mathbf{(\Longrightarrow)}$\ 
  Let $g(x) := \int_{1}^{x} \frac{f(t)}{t}\,\mathrm{d}t$. Since $g(x)\asymp f(x)$ and $g'(x) = f(x)/x$, it follows that
  \[ \frac{g'(x)}{g(x)} \asymp \frac{1}{x}. \]
  Taking $\lambda >1$, integrating from $x$ to $\lambda x$ yields $\log (g(\lambda x)/g(x)) \asymp \log \lambda$, so there are $\vartheta, \eta > 0$ such that $\lambda^{\theta} g(x) \leq g(\lambda x) \leq \lambda^{\eta} g(x)$. Hence,
  \begin{equation}
   r \lambda^{\theta} f(x) \leq f(\lambda x) \leq s \lambda^{\eta} f(x) \label{eqasymp}
  \end{equation}
  for some $r,s>0$. For $0<\lambda <1$, simply take $x= \lambda^{-1} y$ in \eqref{eqasymp}, concluding (i).
  
  For (ii), taking $\lambda = y/x$ in \eqref{eqasymp} yields $r (y/x)^{\theta} f(x) \leq f(y)$, and so
  \[ \dfrac{f(x)}{x^{\vartheta}} \leq M \dfrac{f(y)}{y^{\vartheta}} \]
  for $M = 1/r$.\medskip
  
  \noindent
  $\mathbf{(\Longleftarrow)}$\ Given $x\in\R_{>1}$, let $J\geq 1$ be the smallest integer such that $2^{J} \geq x$. Since $f(2x)\asymp f(x)$, by (i) we have $f(2^{J-1}) \asymp f(x) \asymp f(2^{J})$ (cf. Remark \ref{unifcn}). Moreover,
  \begin{equation*}
   \int_{1}^{x} \frac{f(t)}{t}\,\mathrm{d}t = \sum_{j=0}^{J-2} \int_{2^{j}}^{2^{j+1}} \frac{f(t)}{t}\,\mathrm{d}t + \int_{2^{J-1}}^{x} \frac{f(t)}{t}\,\mathrm{d}t \asymp \sum_{j=1}^{J} f(2^{j}).
  \end{equation*}
  By (ii), $f(2^{\ell}) \leq M 2^{-(J-\ell)\vartheta} f(2^{J})$ for $\ell \geq \ell_0$, where $\ell_0$ is such that $2^{\ell_0} > x_0$. Thus,
  \begin{equation*}
   f(2^{J}) \leq \sum_{\ell=\ell_0}^{J} f(2^{\ell}) \leq M\, f(2^{J}) \sum_{\ell=\ell_0}^{J} \frac{1}{2^{(J-\ell)\vartheta}} \ll f(2^J),
  \end{equation*}
  so $\int_{1}^{x} \frac{f(t)}{t}\,\mathrm{d}t \asymp f(2^{J}) \asymp f(x)$.
 \end{proof}
 
 \begin{rem}[Uniform convergence]\label{unifcn}
  Let $f:\R_{>0}\to\R_{>0}$ be a measurable function satisfying $f(\lambda x) \asymp_{\lambda} f(x)$ for every $\lambda >0$. Then, for every real $\Lambda > 1$ we have
  \[ 0 < \liminf_{x\to\infty} \inf_{\lambda\in [1,\Lambda]} \frac{f(\lambda x)}{f(x)} \leq \limsup_{x\to\infty} \sup_{\lambda\in [1,\Lambda]} \frac{f(\lambda x)}{f(x)} < \infty \] 
  (cf. BGT \cite[Theorem 2.0.8]{bingham89}). This implies, for instance, that for every fixed $\eps >0$, $f(k) \asymp_{\eps} f(n)$ uniformly for $n\geq 1$ and $\eps n \leq k \leq \eps^{-1} n$ (i.e., the implied constant depends only on $\eps$).
 \end{rem}

%%%%%%%%%%%%
\subsection{Exact solutions: main lemma}
 We say that a solution to $b_1k_1+ \cdots +b_{\ell}k_{\ell} = n$ is \emph{exact} if the $k_i$'s are pairwise distinct. Define the \emph{exact representation function}
 \begin{equation}
  \rho_{\A, \ell}(n) = \rho^{(b_1,\ldots,b_{\ell})}_{\A, \ell}(n) := \sum_{\substack{(k_1,\ldots,k_{\ell}) \in \Z_{\geq 0}^{\ell} \\ b_1k_1 + \cdots + b_{\ell}k_{\ell} = n\\ k_i\text{s distinct}}} \mathbbm{1}_{\A}(k_1)\cdots\mathbbm{1}_{\A}(k_{\ell}). \label{defexact}
 \end{equation}
 This function is more amenable to probabilistic methods, since it is a sum of products of $\ell$ independent variables. If $(k_1,\ldots,k_{\ell})$ is a solution to $b_1k_1 + \cdots + b_{\ell}k_{\ell} = n$ with $k_1=k_2$, but $k_2\neq \cdots \neq k_{\ell}$, then $(k_2,\ldots,k_{\ell})$ is an exact solution to the equation 
 \[ (b_1+b_2)k_{2} + b_3k_3 + \cdots + b_{\ell}k_{\ell} = n. \]
 Similarly, every non-exact solution of length $\ell$ yields an exact solution to an equation of smaller length. More precisely, we have
 \begin{equation}
  r_{\A,\ell}(n) = \rho_{\A,\ell}(n) + \sum_{(c_1,\ldots,c_{t})} \rho_{\A,t}^{(c_1,\ldots,c_{t})}(n), \label{excnon}
 \end{equation}
 where the sum runs through the $(c_1,\ldots,c_{t})$, $t< \ell$ that generate equations $c_1k_1+\cdots +c_{t}k_{t} = n$ which are produced by non-exact solutions to $b_1k_1 +\cdots +b_{\ell}k_{\ell}$. Note that $\max_{(c_1,\ldots,c_{t})} \max\{c_i\} \leq b_1+\cdots+b_{\ell}$.
 
 Lemma \ref{fundLem} uses the strategy of dividing the sum over solutions to $b_1k_1+\cdots + b_{\ell} k_{\ell}$ into dyadic intervals by Vu \cite[Lemma 3.3]{vvu00wp}.

 \begin{lem}\label{condB}
  Let $\ell\geq 1$. For every $P_1,\ldots,P_{\ell} > 0$, the number of integer solutions $(k_1,\ldots, k_{\ell}) \in \Z_{\geq 0}^{\ell}$ to
  \[ b_1 k_1 + \cdots + b_{\ell} k_{\ell} = n, \]
  with each $k_j \leq P_j$, is $O_{\ell,b_1,\ldots,b_{\ell}}(\frac{1}{n} P_1\cdots P_{\ell} )$.
 \end{lem}
 \begin{proof}
  If $P_{i}< n/(\max_{i\leq \ell} b_i)\ell$ for every $i$, then there are no solutions $(k_1,\ldots,k_{\ell})$ with $0\leq k_i\leq P_i$, since the sum is $< n$, and the statement is true. So suppose there is some $1\leq j\leq \ell$ for which $P_j \geq n/(\max_{i\leq \ell} b_i)\ell$.
  
  There are at most $P_1\cdots P_{j-1} P_{j-1} \cdots P_\ell$ possible values $b_1k_1 + \cdots + b_{j-1}k_{j-1} + b_{j+1}k_{j+1} + \cdots + b_{\ell} k_{\ell}$ can assume, and for each, there is at most one value of $k_j$ that makes $b_1 k_1 + \cdots + b_{\ell} k_{\ell} = n$. Thus, the number of solutions is $O(P_1\ldots P_{j-1}P_{j+1}\cdots P_\ell)$, and since $P_j/n \geq 1/(\max_{i\leq \ell} b_i)\ell$, this is $O_{\ell,b_1,\ldots,b_{\ell}}(\frac{1}{n} P_1\cdots P_\ell )$.
 \end{proof}

 \begin{lem}[Main lemma]\label{fundLem}
  For every $1\leq \ell \leq h$,
  \[ \frac{\min\{cf(n),n\}^{\ell}}{n}\,\mathbbm{1}_{\gcd(b_1,\ldots, b_{\ell})\mid n} \ll \E(r_{\A,\ell}(n)) \ll c^{\ell} \frac{f(n)^{\ell}}{n}\,\mathbbm{1}_{\gcd(b_1,\ldots, b_{\ell})\mid n}, \]
  where the implied constants do not depend on $c$.
 \end{lem}
 \begin{proof}
  Write $f_{c}(x):= \min\{cf(x),x\}$. By \eqref{excnon}, $r_{\A,\ell}(n)$ can be written as $\rho_{\A,\ell}(n)$ plus contributions coming from exact solutions of finitely many equations of smaller length (with coefficients bounded in terms of $b_1+\cdots+b_\ell$). Therefore it is enough to prove that for every $1\leq t\leq \ell$ and every $(c_1,\ldots,c_t)\in \{1,\ldots,b_1+\cdots+b_{\ell}\}^t$ one has
  \[ \frac{f_{c}(n)^{t}}{n}\, \mathbbm{1}_{\gcd(c_1,\ldots, c_{t})\mid n} \ll \E(\rho^{(c_1,\ldots,c_t)}_{\A,t}(n)) \ll c^{t} \frac{f(n)^{t}}{n}\, \mathbbm{1}_{\gcd(c_1,\ldots, c_{t})\mid n}. \]

  We begin with the lower bound. Fix $t$ and $(c_1,\ldots,c_t)$. Assume that $\gcd(c_1,\ldots,c_t)\mid n$. The number of exact solutions of $c_1k_1+\cdots+c_tk_t=n$ is $\geq \delta n^{t-1}$ for some $\delta>0$. Indeed, the total number of solutions is $\asymp n^{t-1}$, while the non-exact solutions are solutions to equations of length $\leq t-1$, hence contribute $O(n^{t-2})$.

  We now discard solutions with a very small variable. By Lemma \ref{condB}, there exists $C>0$ such that for every $\eps>0$ the number of solutions with $k_j \leq \eps n$ for some $j$ is $\leq C\eps^{t} n^{t-1}$. Choose $\eps>0$ so that $C\eps^{t}<\delta/2$. Then the number of exact solutions with $k_j>\eps n$ for all $j$ is $\gg n^{t-1}$.

  For such solutions, Remark \ref{unifcn} gives $f(k)\geq M f(n)$ for $\eps n \leq k\leq n$, where $M=M_{\eps}>0$. Hence, for $\eps n \leq k\leq n$,
  \[ \frac{cf(k)}{k} \geq M \frac{cf(n)}{n} \geq M \frac{f_c(n)}{n}. \]
  Moreover, since $k\geq \eps n$ we have $k^{-1} \geq n^{-1}$ and also $1 \geq \eps f_c(n)/n$ (because $f_c(n)\leq n$), so $\frac{f_c(k)}{k} \geq M \frac{f_c(n)}{n}$, with $M$ independent of $c$. Therefore
  \begin{align*}
   \E(\rho^{(c_1,\ldots,c_t)}_{\A, t}(n))
   &\geq \sum_{\substack{(k_1, \ldots, k_{t}) \in \Z_{\geq 0}^t \\ c_1 k_1 + \cdots + c_{t} k_{t} = n \\ k_i\text{s distinct} \\ \forall j,\, k_j > \eps n}} \frac{f_{c}(k_1)}{k_1}\cdots \frac{f_{c}(k_{t})}{k_{t}} \\
   &\gg n^{t-1} \bigg(\frac{f_{c}(n)}{n}\bigg)^{t}\, \mathbbm{1}_{\gcd(c_1,\ldots, c_{t})\mid n}
   = \frac{f_{c}(n)^{t}}{n}\, \mathbbm{1}_{\gcd(c_1,\ldots, c_{t})\mid n}.
  \end{align*}

  We turn to the upper bound. Partition the domain into dyadic boxes. Let $\mathcal{P}$ be the set of all $t$-tuples $\mathbf{p} = (P_1,\ldots, P_{t})$ with $P_j\in \{1,2, 4, \ldots, 2^J\}$, where $J$ is the smallest integer such that $2^J \geq n$, and write
  \[ \sigma_{\mathbf{p}} := \sum_{\substack{(k_1, \ldots, k_{t}) \in \Z_{\geq 0}^t \\ c_1 k_1 + \cdots + c_{t} k_{t} = n\\ \forall j,\, \frac{P_j}{2} \leq k_j < P_j}} \frac{f(k_1)}{k_1}\cdots \frac{f(k_{t})}{k_{t}}. \]
  Since $f_c(k)\leq c f(k)$, we have $\E(\rho^{(c_1,\ldots,c_t)}_{\A\setminus\{0\}, t}(n)) \leq c^{t} \sum_{\mathbf{p}\in\mathcal{P}} \sigma_{\mathbf{p}}$. By Lemma \ref{condB}, the number of terms in $\sigma_{\mathbf{p}}$ is $O(\frac{1}{n} P_1\cdots P_{t})$, and for $k_j\in [P_j/2,P_j)$ we have $f(k_j)\asymp f(P_j)$ by Remark \ref{unifcn}. Consequently,
  \begin{align*}
   \E(\rho^{(c_1,\ldots,c_t)}_{\A\setminus\{0\}, t}(n))
   &\ll c^{t} \sum_{\mathbf{p}\in\mathcal{P}} \frac{1}{n}P_1\cdots P_{t}\, \frac{f(P_1)}{P_1} \cdots \frac{f(P_{t})}{P_{t}}\, \mathbbm{1}_{\gcd(c_1,\ldots, c_{t})\mid n} \\
   &\ll c^{t} \frac{1}{n}\big( f(1) + f(2) + f(4) + \cdots + f(2^J) \big)^{t}\, \mathbbm{1}_{\gcd(c_1,\ldots, c_{t})\mid n} \\
   &\ll c^{t} \frac{1}{n} f(2^J)^{t}\, \mathbbm{1}_{\gcd(c_1,\ldots, c_{t})\mid n}
   \asymp c^{t} \frac{f(n)^{t}}{n}\, \mathbbm{1}_{\gcd(c_1,\ldots, c_{t})\mid n}.
  \end{align*}

  Finally, to pass from $\rho_{\A\setminus\{0\},t}$ to $\rho_{\A,t}$ (allowing possible zero variables), note that if a $t$-tuple contributing to $\rho_{\A,t}(n)$ has some $k_j=0$, then it contributes to $\rho^{(d_1,\ldots,d_{t-1})}_{\A\setminus\{0\}, t-1}(n)$ for a coefficient subtuple $(d_1,\ldots,d_{t-1})$ of $(c_1,\ldots,c_t)$. Hence
  \[ \rho^{(c_1,\ldots,c_t)}_{\A,t}(n) \leq \rho^{(c_1,\ldots,c_t)}_{\A\setminus\{0\},t}(n) + \sum_{(d_1,\ldots,d_{t-1})} \rho^{(d_1,\ldots,d_{t-1})}_{\A\setminus\{0\}, t-1}(n), \]
  where the sum runs over all subtuples of size $t-1$. By induction on $t$, this shows that it is enough to bound $\rho^{(c_1,\ldots,c_t)}_{\A\setminus\{0\},t}(n)$, and the proof is complete.
 \end{proof}

 \begin{xrem}[Equations of smaller length]
  For $2\leq \ell \leq h$, write
  \begin{equation}
   r^{*}_{\A,\ell}(n) := \max_{1\leq i_1<\ldots< i_{\ell}\leq h} \#\{(k_1,\ldots, k_{\ell})\in \A^{\ell} ~|~ b_{i_1}k_1+\cdots+ b_{i_{\ell}}k_{\ell} = n\} \label{rstar}
  \end{equation}
  for the maximum among a choice of $b_{i_1},\ldots,b_{i_{\ell}}$ of the number of solutions to $b_{i_1}k_1+\cdots+ b_{i_{\ell}}k_{\ell} = n$. Since there is only a finite number of choices of $b_{i_1},\ldots,b_{i_{\ell}}$, Lemma \ref{fundLem} implies that $\E(r^{*}_{\A,\ell}(n)) \ll c^{\ell}f(n)^{\ell}/n$.
  
  Whenever we add a ``$*$'' to a representation function (e.g., $\rho^{*}_{\A,\ell}(n)$), we are taking the maximum among a choice of $b_{i_1},\ldots,b_{i_{\ell}}$ in the definition of that function.
 \end{xrem}

%%%%%%%%%%%%%%%%%%%%%
\subsection{Theorem \ref{MT2}: Case \texorpdfstring{$h = 2$}{h = 2}}\label{caseh2}
 Suppose that $(x\log x)^{1/2} \ll f(x) \ll x$. Write
 \[ r_{\A,2}(n) = X_1(n)+X_2(n)+O(1), \]
 where $X_1(n)$ counts solutions to $b_1 k_1 + b_2 k_2 = n$ with $k_1<k_2$, and $X_2(n)$ counts those with $k_1>k_2$. Each $X_i(n)$ is a sum of independent boolean random variables $\mathbbm{1}_{\A}(k_1)\mathbbm{1}_{\A}(k_2)$. By Lemma \ref{fundLem}, we have $\E(X_1(n)+X_2(n)) \gg \min\{c f(n), n\}^2/n$,
 and hence, for at least one $i=i(n)\in\{1,2\}$,
 \[ \E(X_i(n)) \gg \min\{c f(n), n\}^2/n \geq c^2\, d \log n \]
 for some $d>0$. For such $i$, by Chernoff's inequality \cite[Theorem 1.8]{taovu06}, we obtain
 \begin{align*}
  \Pr(X_i(n) \leq \tfrac{1}{2} \E(X_i(n))) &\leq 2e^{-\frac{1}{16} \E(X_i(n))} \\
  &\leq 2e^{-\frac{1}{16} c^{2} d \log n} \leq n^{-2}
 \end{align*}
 for large enough $c$. By the Borel--Cantelli lemma, $r_{\A,2}(n) \stackrel{\textnormal{a.s.}}{\gg} f(n)^2/n$.

 Conversely, by Lemma \ref{fundLem}, there exists $C_0>0$ such that $\E(X_i(n)) \leq C_0 f(n)^2/n$ for $i=1,2$. For $C>3C_0$, Chernoff's inequality \cite[Theorem 1.8]{taovu06} gives
 \[ \Pr(X_i(n) \geq C f(n)^2/n) \leq 2e^{-(\frac{C-C_0}{2}) f(n)^2/n}, \]
 so since $f(n)^2/n\gg \log n$, taking $C$ large enough yields $\Pr(X_i(n) \geq C f(n)^2/n)\leq n^{-2}$ for $i=1,2$. Applying the Borel--Cantelli lemma to both $X_1$ and $X_2$, we conclude that $r_{\A,2}(n)\stackrel{\textnormal{a.s.}}{\ll} f(n)^2/n$, completing the proof.\hfill$\square$

%%%%%%%%%%%%%%%%%%%%%%%%%
\subsection{\texorpdfstring{$\delta$}{delta}-small and \texorpdfstring{$\delta$}{delta}-normal solutions}
 As in Vu \cite{vvu00wp}, in order to estimate $r_{\A, h}(n)$, we separate the solutions being counted into \emph{small} and \emph{normal}, depending on a parameter $\delta$, as follows: For $0<\delta < 1$, define
 \begin{equation}
 \begin{aligned}
  r^{(\delta\textnormal{-small})}_{\A, \ell}(n) &:= \sum_{\substack{(k_1, \ldots, k_{\ell}) \in \Z_{\geq 0}^{\ell} \\ b_1k_1 + \cdots + b_{\ell}k_{\ell} = n \\ \exists j ~|~ k_j < n^{\delta}}} \mathbbm{1}_{\A}(k_1)\cdots \mathbbm{1}_{\A}(k_{\ell}), \\
  r^{(\delta\textnormal{-normal})}_{\A, \ell}(n) &:= \sum_{\substack{(k_1, \ldots, k_{\ell}) \in \Z_{\geq 0}^{\ell} \\ b_1k_1 + \cdots + b_{\ell}k_{\ell} = n \\ k_1,\ldots, k_{\ell} \geq n^{\delta}}} \mathbbm{1}_{\A}(k_1)\cdots \mathbbm{1}_{\A}(k_{\ell}),
 \end{aligned}\label{deltans}
 \end{equation} 
 so that $r_{\A, \ell}(n) = r^{(\delta\textnormal{-small})}_{\A, \ell}(n) + r^{(\delta\textnormal{-normal})}_{\A, \ell}(n)$. Both $\rho^{(\delta\textnormal{-small})}_{\A, \ell}(n)$ and $\rho^{(\delta\textnormal{-normal})}_{\A, \ell}(n)$ are defined similarly. We show that $\delta$-small solutions are, on average, few.

 %%%% 
 \begin{lem}\label{expdsml}
  Let $\vartheta = \vartheta_f$ be as in Lemma \ref{orpiCHAR} (ii). Then, for $2\leq \ell\leq h$ and every $0 < \delta < 1$, we have
  \[ \E(r^{(\delta\textnormal{-small})}_{\A, \ell}(n)) \ll c^{\ell}\, n^{-(1-\delta)\vartheta} \frac{f(n)^{\ell}}{n}. \]
 \end{lem}
 \begin{proof}
  Let $r^{*}$ be as in \eqref{rstar}. We have
  \begin{align*}
   \E(r^{(\delta\textnormal{-small})}_{\A, \ell}(n)) &\leq \sum_{j=1}^{\ell} \sum_{k_j \leq n^{\delta}} \E\big(r^{*}_{\A,\ell-1}(n-b_j k_j) \mathbbm{1}_{\A}(k_j) \big) \\
   &\leq c\sum_{j=1}^{\ell} \sum_{k_j \leq n^{\delta}} \frac{f(k_j)}{k_j}\, \E\big(r^{*}_{\A,\ell-1}(n-b_j k_j) ~|~ \mathbbm{1}_{\A}(k_j) = 1 \big).   
  \end{align*}
  Using that
  \begin{align*}
   \E\big(r^{*}_{\A,\ell-1}(n-b_j k_j) ~|~ \mathbbm{1}_{\A}(k_j) = 1 \big) &\leq \sum_{t=1}^{\ell-1} \sum_{b=1}^{\ell(\max b_i)} \E\big(r^{*}_{\A,\ell-t}(n-b k_j) \big)   
  \end{align*}
  we obtain by Lemma \ref{fundLem} that
  \begin{align*}
   \E(r^{(\delta\textnormal{-small})}_{\A, \ell}(n)) &\ll c^{\ell}\frac{f(n)^{\ell-1}}{n}\, \sum_{k\leq n^{\delta}} \frac{f(k)}{k} \\
   &\asymp c^{\ell} f(n^{\delta}) \frac{f(n)^{\ell-1}}{n}.
  \end{align*}
  Since there is $\vartheta = \vartheta_{f} >0$ for which $f(n^{\delta}) = f(n^{-(1-\delta)}\, n) \ll n^{-(1-\delta)\vartheta} f(n)$ by Lemma \ref{orpiCHAR}, the lemma follows.
 \end{proof}

 %%%%%%%%%%%
\subsection{Maxdisfam of representations}
 Let $\widehat{r}_{\A,\ell}(n)$ denote the maximum size of a disjoint family (abbreviated \emph{disfam}) of solutions $R = (k_1,\ldots,k_{\ell}) \in \A$ of $n$. Thus, $\widehat{r}_{\A,\ell}(n) = |\mathcal{M}|$ for some maximal disjoint family (abbreviated \emph{maxdisfam}) of representations. This means that for every solution $R$, there is a solution $S\in \mathcal{M}$ such that $S\cap R\neq \varnothing$. Hence, for $r^{*}$ as in \eqref{rstar},
 \begin{equation}
  r_{\A,\ell}(n) \leq \sum_{j=1}^{\ell} \sum_{\substack{k \in S \\ S\in \mathcal{M}}} r^{*}_{\A,\ell-1}(n-b_j k) \leq \ell\cdot \ell! \, \widehat{r}_{\A,\ell}(n) \left(\max_{k\leq n} r^{*}_{\A,\ell-1}(k) \right). \label{mchdeg}
 \end{equation}
 Whenever we add a `` $\widehat{\ }$ '' to a representation function, we are taking the size of a maximum disjoint family of representations counted by that function: e.g., $\widehat{r}^{*}_{\A,\ell}$, $\widehat{\rho}_{\A,\ell}$, $\widehat{\rho}^{(\delta\text{-small})}_{\A,\ell}$. We state the next lemma in sufficient generality to cover most use cases.
 
 \begin{lem}\label{prodHat}
  For every $2\leq \ell\leq h$, we have
  \[ \mathrm{R}_{\A,\ell}(n) \ll \widehat{\mathrm{R}}_{\A,\ell}(n) \left(\max_{k\leq n} \widehat{r}^{*}_{\A,\ell-1}(k) \right) \cdots \left(\max_{k\leq n} \widehat{r}^{*}_{\A,2}(k) \right), \]
  where $\mathrm{R} = r$, $r^{*}$, $\rho$, $\rho^{(\delta\textnormal{-small})}$.
 \end{lem}
 \begin{proof}
  Given that \eqref{mchdeg} applies to $\mathrm{R}$, we keep applying the same bound to $r^{*}_{\A,\ell-t}$ ($1\leq t\leq \ell-2$), obtaining
  \[ \mathrm{R}_{\A,\ell}(n) \ll \widehat{\mathrm{R}}_{\A,\ell}(n) \left(\max_{k\leq n} \widehat{r}^{*}_{\A,\ell-1}(k) \right) \cdots \left(\max_{k\leq n} \widehat{r}^{*}_{\A,3}(k) \right) \left(\max_{k\leq n} r^{*}_{\A,2}(k) \right). \]
  Since $r^{*}_{\A,2}(k) \leq 2\widehat{r}_{\A,2}^{*}(k)$, the conclusion follows.
 \end{proof}

 Lemma \ref{prodHat} will be used together with the following lemma \cite[Lemma 1]{erdtet90}:
 
 \begin{lem}[Disjointness lemma]\label{disjlm}
  Let $\mathscr{E} = \{E_1,E_2,\ldots\}$ be a family of events, and define $S := \sum_{E\in\mathscr{E}} \mathbbm{1}_E$. If $\E(S) < \infty$, then for every $k\in\Z_{\geq 1}$ we have
  \[ \Pr(\exists\mathcal{D}\subseteq \mathscr{E} \textnormal{ disfam}\ |\mathcal{D}| = k) \leq \sum_{\substack{\mathcal{J}\subseteq \mathscr{E}\textnormal{ disfam} \\ |\mathcal{J}| = k}} \Pr\bigg( \bigwedge_{E\in\mathcal{J}} E\bigg) \leq \frac{\E(S)^{k}}{k!}. \]
 \end{lem}
 \begin{proof}
  \begin{align*}
   \sum_{\substack{\mathcal{J}\subseteq \mathscr{E}\textnormal{ disfam} \\ |\mathcal{J}| = k}} \Pr\bigg( \bigwedge_{E\in\mathcal{J}} E\bigg) = \sum_{\substack{\mathcal{J}\subseteq \mathscr{E}\textnormal{ disfam} \\ |\mathcal{J}| = k}} \prod_{E\in\mathcal{J}} \Pr(E) &\leq \frac{1}{k!}\bigg( \sum_{E\in \mathscr{E}} \Pr(E) \bigg)^{k} = \frac{\E(S)^{k}}{k!}. \qedhere
  \end{align*}
 \end{proof}
 
 \begin{xrem}
 If there exists a maxdisfam of size greater than $k$, then in particular there exists a disfam of size $k$. Thus, the form we will apply this lemma is as follows: Since $k!\geq k^k e^{-k}$,
  \begin{align*}
   \Pr(\exists \text{maxdisfam of size} \geq x) &\leq \bigg(\frac{e\,\mathbb{E}}{\lceil x\rceil}\bigg)^{\lceil x\rceil}\qquad  \text{(for real $x\geq 1$)} \\
   &\leq \bigg(\frac{e\,\mathbb{E}}{x}\bigg)^{x} \phantom{\bigg(\frac{e\,\mathbb{E}}{\lceil x\rceil}\bigg)^{\lceil x\rceil}}\negphantom{\bigg(\frac{e\,\mathbb{E}}{x}\bigg)^{x}} \qquad \text{(for real $x\geq 1+\E$)}. 
  \end{align*}
 \end{xrem}

%%%%%%%%%%%%%%%%%%%%%%%%%% 
%%%%%%%%%%%%%%%%%%%%%%%%%%
%%%%%%%%%%%%%%%%%%%%%%%%%%
\section{Asymptotic case}%: \texorpdfstring{$(x\log x)^{1/h} \ll f \ll x^{1/(h-1)}$}{[equation]}} 
 For this section, fix $h\geq 2$, let $F(x) = x^{\kappa}\phi(x)$ for some slowly varying function $\phi$ and some real $0\leq \kappa \leq h-1$, and define
 \begin{equation}
  f(x) := (xF(x))^{1/h} = x^{(1+\kappa)/h} \phi(x)^{1/h}. \label{fthisec}
 \end{equation}
 It is sufficient to work with $f$ satisfying
 \[ \frac{f(x)}{(x\log x)^{1/h}} \xrightarrow{x\to\infty} \infty\quad \text{ and }\quad cf(x) \leq x. \]
 
 \subsection{Concentration of boolean polynomials}\label{secCBP1}
 We will prove that $r_{\A,h}$ strongly concentrates around its mean using the strategy of Vu \cite{vvu00wp, vvu00mc}. Precisely, let $n\in \Z_{\geq 1}$, and take $v_1$, $\ldots$, $v_{n}$ to be independent, not necessarily identically distributed, $\{0,1\}$-random variables. A \emph{boolean polynomial} is a multivariate polynomial
 \[ Y(v_1,\ldots,v_n) = \sum_{i} c_i I_i \in \R[v_1,\ldots,v_{n}], \]
 where the $I_i$s are monomials: products of some of the $v_k$s. We say that $f$ is
 \begin{itemize}
  \item \emph{positive} if $c_i \in \R_{>0}$ for every $i$;
  
  \item \emph{simple} if the largest exponent of $v_i$ in a monomial is $1$ for every $i$;
  
  \item \emph{homogeneous} if every monomial has the same degree;
  
  \item \emph{normal} if $0\leq c_i \leq 1$ for every $i$, and the free coefficient of $Y$ is $0$.
 \end{itemize}
 For a non-empty multiset\footnote{A \emph{multiset} is a set that allows multiple instances of an element.} $S\subseteq \{v_1,\ldots,v_n\}$, define $\partial_S :=  \prod_{v\in S} \partial_{v}$, where $\partial_{v}$ is the partial derivative in $v$. For example: if $S = \{1,1,2\}$, then $\partial_{S} (v_1^3 v_2 v_3 + 3v_1^5) = 6v_1 v_3$. Define
 \[ \E_j(Y) := \max_{\substack{S\subseteq \{v_1,\ldots,v_n\} \\ \text{multiset},\, |S| = j}} \E(\partial_S Y), \qquad \E'(Y) = \max_{j\geq 1}\, \E_j(Y). \]
 We will need two concentration results:
 
 \begin{thm}[Kim--Vu \cite{kimvvu00}, 2000]\label{kimvu}
  Let $d\geq 1$, and $Y(v_1,\ldots,v_n)$ is a positive, simple boolean polynomial of degree $d$. Write $E' := \E'(Y)$ and $E := \max\{\E(Y), E'\}$. Then, for any real $\lambda \geq 1$, we have
  \[ \Pr\big(|Y-\E(Y)| > 8^{d}\sqrt{d!}\, \lambda^d\, (E'E)^{1/2} \big) \ll_{d} n^{d-1}e^{-\lambda} \]
 \end{thm}

 In applications, we will always take $\lambda = (d+1)\log n$, so Kim--Vu's inequality will be useful when $1\ll \E' \ll \E/(\log n)^{2d}$. To deal with the cases with small expectation, we will use a corollary of another theorem of Vu \cite[Theorem 1.4]{vvu00mc}:
 
 \begin{thm}[Vu, 2000]\label{vuORI}
  Let $d\geq 2$, and $Y(v_1,\ldots,v_n) = \sum_i c_i I_i$ be a simple, homogeneous, normal boolean polynomial of degree $d$. Then, for any $\alpha, \beta \in \R_{>0}$, there exists a constant $K = K(d,\alpha,\beta)$ such that: If $\E_1(Y)$, $\ldots$, $\E_{d-1}(Y) \leq n^{-\alpha}$, then for any real $0< \lambda \leq \E(Y)$ we have
  \[ \Pr(|Y -\E(Y)| \geq (\lambda \E(Y))^{1/2}) \leq 2d\, e^{-\lambda/16 dK} + n^{-\beta}. \]
 \end{thm}
 
 \begin{xrem}
  The random variable $r_{\A,\ell}(n)$ can be seen as a boolean polynomial
  \[ r_{\A,\ell}(n) = Y(v_{1}, \ldots, v_n),  \]
  where the $v_i$'s are independent $\{0,1\}$-random variables with $\Pr(v_i = 1) = \E(\mathbbm{1}_{\A}(n))$. The expectations of derivatives of $r_{\A,\ell}(n)$ are thus bounded from above by
  \[ O_{\ell}(\E(r^{*}_{\A,\ell-t}(n-k))), \qquad 1 \leq k\leq n.\]
  for $1\leq t\leq \ell$.
 \end{xrem}

%%%%%%%%%%%%%%%%%%%%%%%%%%%%%
\subsection{Expectation of \texorpdfstring{$r_{\A,h}(n)$}{r\_A,h(n)}}
 In the space \eqref{presp} defined by $f$ as in \eqref{fthisec}, we can get more precise estimates for $\E(r_{\A,\ell}(n))$. Note that by the strong law of large numbers,
 \begin{align}
  |\A \cap [1,x]| &\stackrel{\text{a.s.}}{\sim} c\int_{1}^{x} t^{\frac{1+\kappa}{h}-1} \phi(t)^{1/h}\,\mathrm{d}t \nonumber \\
  &= c\bigg(\int_{1/x}^{1} u^{\frac{1+\kappa}{h}-1} \bigg(\frac{\phi(ux)}{\phi(x)}\bigg)^{1/h}\,\mathrm{d}u\bigg) x^{(1+\kappa)/h} \phi(x)^{1/h} \nonumber \\
  &\sim c\frac{h}{1+\kappa}x^{(1+\kappa)/h} \phi(x)^{1/h}. \label{slln1}
 \end{align}
 The last line is obtained as follows: by Potter bounds (cf. BGT \cite[Theorem 1.5.6 (i)]{bingham89}), for every $\delta > 0$ there is $C = C_{\delta}$ such that, for large $x\geq x_{\delta}$ and $C/x \leq u \leq 1$, we have $\phi(ux)/\phi(x) \leq 2u^{-\delta}$. Choosing $\delta < \frac{1+\kappa}{2}$, we split the integral $\int_{1/x}^{1} = \int_{1/x}^{C/x} + \int_{C/x}^{1}$. Since the definition of slowly varying implies that
 \[ u^{\frac{1+\kappa}{h}-1} \bigg(\frac{\phi(ux)}{\phi(x)}\bigg)^{1/h}\mathbbm{1}_{(C/x,1]} \xrightarrow{x\to\infty} u^{\frac{1+\kappa}{h}-1} \mathbbm{1}_{(0,1]}, \]
 the dominated convergence theorem yields that $\int_{C/x}^{1} u^{\frac{1+\kappa}{h}-1} (\frac{\phi(ux)}{\phi(x)})^{1/h}\,\mathrm{d}u \to \int_{0}^{1} u^{\frac{1+\kappa}{h}-1}$ as $x\to \infty$. On the other hand, since $\phi(x) = x^{o(1)}$, the term $\int_{1/x}^{C/x}$ vanishes.
 
 We are going to show that:
 
 \begin{lem}\label{expgood}
  We have
  \[ \E(r_{\A,h}(n))\sim\, c^{h} \frac{\Gamma(\frac{1+\kappa}{h})^{h}}{\Gamma(1+\kappa)}\,\frac{F(n)}{(b_1\cdots b_{h})^{\frac{1+\kappa}{h}}}. \]
 \end{lem}
 
 The proof also works for $F=\log$, a fact which will be used in Section \ref{secwi} (in fact, we only use that $F(x) \to \infty$ as $x\to\infty$). We start with the following lemma:
 
 \begin{lem}\label{basiclem}
  Let $L \geq 1$ and $1\leq  r \leq L$ be integers. For any real numbers $\alpha \geq \beta > 0$, we have
  \[ \sum_{\substack{m=1 \\ m \equiv r \pmod{L}}}^{n-1} m^{\alpha - 1} (n - m)^{\beta - 1} = \frac{\Gamma(\alpha)\Gamma(\beta)}{\Gamma(\alpha+\beta)} \frac{n^{\alpha+\beta-1}}{L} + O_{\alpha,\beta,L}(n^{\alpha - 1} + n^{\alpha+\beta-2}). \]
 \end{lem}
 \begin{proof}
  Define
  \[ \gamma_n(t) := (Lt + r)^{\alpha - 1} (n - r - Lt)^{\beta - 1}, \quad \text{for } t \in \left[0, \frac{n - r}{L} \right], \]
  so that the sum we want to estimate equals
  \[ S(n) = \sum_{k = 0}^{K} \gamma_n(k), \quad \text{where } K := \left\lfloor \frac{n - r - 1}{L} \right\rfloor. \]
  The function $x \mapsto x^{\alpha - 1}(n - x)^{\beta - 1}$ has at most one critical point in $(0, n)$, since its logarithmic derivative
  \[ \frac{\mathrm{d}}{\mathrm{d}x} \log\left( x^{\alpha - 1} (n - x)^{\beta - 1} \right) = \frac{\alpha - 1}{x} - \frac{\beta - 1}{n - x}\]
  changes sign at most once. Thus, $\gamma_n(t)$ is unimodal in $(0,K)$, and can be well approximated by the integral:
  \[ S(n) = \sum_{k = 0}^{K} \gamma_n(k) = \int_{0}^{K} \gamma_n(t)\,\mathrm{d}t + O\bigg( \sup_{t \in [0,K]} \gamma_n(t) \bigg). \]
  The maximum of $\gamma_n(t)$ is either attained at the critical point (which can be shown to be of the form $t^{*} \sim cn$ for some constant $c\in(0,1)$ when $\alpha,\beta>1$), giving $O_{\alpha,\beta,L}(n^{\alpha+\beta-2})$, or at the extremes of $[0,K]$, giving $O_{\alpha,\beta,L}(n^{\alpha-1})$ (since $\alpha\geq \beta$).
  
  To compute the integral, we first change variables: let $x = Lt + r$, so that $t = \frac{x - r}{L}$ and $\mathrm{d}t = \frac{\mathrm{d}x}{L}$. Then
  \begin{align*}
   S(n) &= \frac{1}{L} \int_{r}^{r + LK} x^{\alpha - 1}(n - x)^{\beta - 1}\,\mathrm{d}x + O_{\alpha,\beta,L}(n^{\alpha-1}+n^{\alpha+\beta-2})
  \end{align*}
  The integral $\int_{r}^{r + LK} \cdots$ differs from the full interval $[0, n]$ only near the endpoints, where the integral is bounded by
  \[ \bigg(\int_{0}^{r} + \int_{r+LK}^{n}\bigg) x^{\alpha - 1}(n - x)^{\beta - 1}\,\mathrm{d}x \ll n^{\alpha-1} \max\bigg\{\frac{L^{\alpha}}{\alpha}, \frac{L^{\beta}}{\beta}\bigg\}. \]
  Therefore:
  \begin{align*}
   S(n) &= \frac{1}{L} \int_{0}^{n} x^{\alpha - 1}(n - x)^{\beta - 1}\,\mathrm{d}x + O_{\alpha,\beta,L}(n^{\alpha-1}+n^{\alpha+\beta-2}) \\
   &= \frac{n^{\alpha + \beta - 1}}{L} \int_{0}^{1} u^{\alpha - 1}(1 - u)^{\beta - 1}\,\mathrm{d}u + O_{\alpha,\beta,L}(n^{\alpha-1}+n^{\alpha+\beta-2}) \\
   &= \frac{\Gamma(\alpha)\Gamma(\beta)}{\Gamma(\alpha + \beta)} \frac{n^{\alpha + \beta - 1}}{L} + O_{\alpha,\beta,L}(n^{\alpha-1}+n^{\alpha+\beta-2}),
  \end{align*}
  as claimed.
 \end{proof}

 %%%%%%%%%%%%%%%%%%%%%%%%%%%%%%
 \begin{lem}\label{expgoodnp}
  For any real $\omega>0$, we have
  \[ \sum_{\substack{(k_1,\ldots,k_{h}) \in \Z_{\geq 1}^{h} \\ b_1 k_1+\cdots+b_{h} k_{h} = n}} (k_1\cdots k_{h})^{\omega - 1} = \frac{\Gamma(\omega)^{h}}{\Gamma(h\omega)} \frac{n^{h\omega - 1}}{(b_1\cdots b_h)^{\omega}} + O(n^{h\omega-1 - \min\{1,\omega\}}). \]
 \end{lem}
 \begin{proof}
  We will show by induction that for $2\leq \ell\leq h$, we have
  \begin{align}
   S_{\ell}(n) &:= \sum_{\substack{(k_1,\ldots,k_{\ell}) \in \Z_{\geq 1}^{\ell} \\ b_1k_1+\cdots+b_{\ell}k_{\ell} = n}} (k_1\cdots k_{\ell})^{\omega - 1} \nonumber \\
   &= \frac{\Gamma(\omega)^{\ell}}{\Gamma(\ell\omega)} \frac{\gcd(b_1,\ldots b_{\ell})}{(b_1\cdots b_{\ell})^{\omega}} n^{\ell\omega - 1}\, \mathbbm{1}_{\gcd(b_1,\ldots b_{\ell})\mid n}  + O(n^{\max\{(\ell-1)\omega - 1,\,\ell\omega-2\}}). \label{indct}
  \end{align}
  For the case $\ell = 2$, we have
  \begin{align*}
   S_2(n) = \sum_{\substack{(k_1,k_2)\in\Z_{\geq 1}^{2} \\ b_1k_1 + b_2k_2 = n}} (k_1 k_2)^{\omega-1} &= \sum_{\substack{k=1 \\ b_1k\equiv n\pmod{b_2}}}^{n/b_1-1} k^{\omega-1}\bigg(\frac{n-b_1k}{b_2}\bigg)^{\omega-1} \\
   &= \frac{1}{(b_1b_2)^{\omega-1}} \sum_{\substack{m=1 \\ m\equiv 0\pmod{b_1} \\ m\equiv n\pmod{b_2}}}^{n-1} m^{\omega-1}(n-m)^{\omega-1}
  \end{align*}
  The system of congruences $m\equiv 0\pmod{b_1}$, $m\equiv n\pmod{b_2}$ has a solution $m\equiv r\pmod{\mathrm{lcm}(b_1,b_2)}$ if and only if $\gcd(b_1,b_2)\mid n$. Therefore, by Lemma \ref{basiclem},
  \[ S_2(n) = \frac{\Gamma(\omega)^2}{\Gamma(2\omega)} \frac{\gcd(b_1,b_2)}{(b_1b_2)^{\omega}} n^{2\omega-1}\,\mathbbm{1}_{\gcd(b_1,b_{2})\mid n} + O(n^{\max\{\omega-1,\,2\omega-2\}}). \]
    
  Assume by induction that \eqref{indct} holds for some $\ell\geq 2$. Write $G = \gcd(b_1,\ldots,b_{\ell})$. From Lemma \ref{basiclem}, we get
  \begin{align*}
   S_{\ell+1}(n) &= \sum_{\substack{(k_1,\ldots,k_{\ell}) \in \Z_{\geq 1}^{\ell} \\ b_1k_1+\cdots+ b_{\ell}k_{\ell} \leq n \\ b_1k_1+\cdots+ b_{\ell}k_{\ell} \equiv n \pmod{b_{\ell+1}}}} (k_1\cdots k_{\ell})^{\omega - 1} \bigg(\frac{n-(b_1k_1+\cdots +b_{\ell}k_{\ell})}{b_{\ell + 1}}\bigg)^{\omega-1} \\
   &= \sum_{\substack{m=1 \\ m\equiv n \pmod{b_{\ell+1}}}}^{n-1} S_{\ell}(m)\, \bigg(\frac{n-m}{b_{\ell+1}}\bigg)^{\omega-1} \\
   &= \frac{\Gamma(\omega)^{\ell}}{\Gamma(\ell\omega)} \frac{G}{(b_1\cdots b_{\ell})^{\omega}b_{\ell+1}^{\omega-1}} \sum_{\substack{m=1 \\ m\equiv 0\pmod{G} \\ m\equiv n \pmod{b_{\ell+1}}}}^{n-1} m^{\ell\omega-1} (n-m)^{\omega-1} \,+ \\
   &\hspace{13em}+ O\bigg(\sum_{m=1}^{n-1} (m^{\max\{(\ell-1)\omega-1,\,\ell\omega-2\}}) (n-m)^{\omega-1} \bigg).
  \end{align*}
  By Lemma \ref{basiclem}, the error term is $O(n^{\max\{\ell\omega-1,\, (\ell+1)\omega-2\}})$, and applying Lemma \ref{basiclem} (with $\alpha=\ell\omega$ and $\beta=\omega$) to the main sum gives
  \begin{align*}
   S_{\ell+1}(n) &= \frac{\Gamma(\omega)^{\ell+1}}{\Gamma((\ell+1)\omega)} \frac{Gb_{\ell+1}}{(b_1\cdots b_{\ell+1})^{\omega}} \frac{n^{(\ell+1)\omega-1}}{\mathrm{lcm}(G,b_{\ell+1})}\mathbbm{1}_{\gcd(G,b_{\ell+1})\mid n} \\
   &\hspace{+21em}+ O(n^{\max\{\ell\omega-1,\, (\ell+1)\omega-2\}}).
  \end{align*}
  Since $\frac{Gb_{\ell+1}}{\mathrm{lcm}(G,b_{\ell+1})} = \gcd(G,b_{\ell+1}) = \gcd(b_1,\ldots,b_{\ell+1})$, this proves \eqref{indct} for $\ell+1$, concluding the induction.
 \end{proof}

 Let $\phi$ be a slowly varying function. By uniform convergence (cf. BGT \cite[Theorem 1.2.1]{bingham89}), given $0<\mu<1$, for every $\eps>0$ there exists $x_{\mu,\eps}\in \R$ such that, for every $x\geq x_{\mu,\eps}$,
 \begin{equation*}
  \bigg|\frac{\phi(\lambda x)}{\phi(x)} - 1\bigg| < \eps, \qquad \forall \lambda \in [\mu,1].
 \end{equation*}
 Taking $\mu = 1/j$, $j\in\Z_{\geq 1}$, we define $\xi(x) := j$ for $x\in [x_{\frac{1}{j},\frac{1}{j}}, x_{\frac{1}{j+1}, \frac{1}{j+1}})$. This defines a non-decreasing function $\xi(x)\to \infty$ such that
 \begin{equation}
  \frac{\phi(y)}{\phi(x)} \to 1 \text{ uniformly for } y \in \bigg[\frac{x}{\xi(x)}, x\bigg). \label{eps0}
 \end{equation}
 
 We are now ready to prove Lemma \ref{expgood}.
 
 \begin{proof}[Proof of Lemma \ref{expgood}]
  By Lemma \ref{fundLem}, non-exact solutions and solutions containing some $k_i=0$ do not contribute more than $O_c(\frac{f(n)^{h-1}}{n})$, therefore
  \begin{align*}
   \E(r_{\A,h}(n)) = c^{h}\sum_{\substack{(k_1, \ldots, k_{h}) \in \Z_{\geq 1}^{h} \\ b_1k_1 + \cdots + b_{h}k_{h} = n}} \frac{f(k_1)}{k_1}\cdots\frac{f(k_h)}{k_h} + O_c\bigg(\frac{f(n)^{h-1}}{n}\bigg).
  \end{align*}
  Let $\xi(x)$ be as in \eqref{eps0}. Start by separating the sum into
  \[ S_1 := \sum_{\substack{(k_1,\ldots,k_h)\in \Z_{\geq 1}^h \\ b_1k_1+\cdots+b_h k_h = n \\ \exists j \,|\, k_{j}< n/\xi(n)}} \frac{f(k_1)}{k_1}\cdots\frac{f(k_h)}{k_h},\qquad S_2 := \sum_{\substack{(k_1,\ldots,k_h)\in \Z_{\geq 1}^h \\ b_1k_1 + \cdots + b_h k_h = n \\ \forall j,\, k_{j}\geq n/\xi(n)}} \frac{f(k_1)}{k_1}\cdots\frac{f(k_h)}{k_h}. \]
  
  We start with $S_1$. We have $S_1 \leq \sum_{\ell=1}^{h} S_{1,\ell}$, where
  \[ S_{1,j} := \sum_{\substack{(k_1,\ldots,k_h)\in \Z_{\geq 1}^h \\ b_1k_1+\cdots+b_h k_h = n \\ k_{j}< n/\xi(n)}} \frac{f(k_1)}{k_1}\cdots\frac{f(k_h)}{k_h}. \]
  Let $\mathcal{P}$ be the set of all $h$-tuples $\mathbf{p} = (P_1,\ldots, P_{h})$ with $P_1\in \{1,2,4\ldots, 2^{L}\}$, and $P_j\in \{1,2, 4, \ldots, 2^J\}$ ($2\leq j\leq h$), where $L$ (resp. $J$) is the smallest integer for which $2^L \geq n/\xi(n)$ (resp. $2^J \geq n$), and write
  \[ \sigma_{\mathbf{p}} := \sum_{\substack{(k_1,\ldots,k_h)\in \Z_{\geq 1}^h \\ b_1k_1+\cdots+b_h k_h = n \\ \forall j,\, \frac{P_j}{2} \leq k_j < P_j}} \frac{f(k_1)}{k_1}\cdots \frac{f(k_h)}{k_h}. \]
  We have $S_{1,1} \leq \sum_{\mathbf{p}\in\mathcal{P}} \sigma_{\mathbf{p}}$. The number of terms in $\sigma_{\mathbf{p}}$ is $O(\frac{1}{n} P_1\cdots P_h)$ by Lemma \ref{condB}. Moreover, since $f(x) = x^{\frac{1+\kappa}{h}}\phi(x)^{1/h}$ is regularly varying, we have $\sum_{j=0}^{J} f(2^{j}) \ll_{\eps} f(n)\sum_{j=0}^{J} 2^{-j(\frac{1+\kappa}{h} -\eps)} \ll f(n)$. Hence:
  \begin{align*}
   S_{1,1} \leq \sum_{\mathbf{p}\in\mathcal{P}} \sigma_{\mathbf{p}} &\ll \sum_{\mathbf{p}\in\mathcal{P}} \frac{1}{n}(P_1\cdots P_{h})\, \frac{f(P_1)}{P_1} \cdots \frac{f(P_{h})}{P_{h}} \\
   &= \frac{1}{n} \sum_{\mathbf{p}\in\mathcal{P}} f(P_1)\cdots f(P_h) \\
   &\ll \frac{1}{n}\Bigg( \sum_{j=0}^{L} f(2^{j})\Bigg)\Bigg( \sum_{j=0}^{J} f(2^{j})\Bigg)^{h-1} \\
   &\ll \frac{f(n/\xi(n))f(n)^{h-1}}{n} = o\bigg(\frac{f(n)^{h}}{n}\bigg).
  \end{align*}  
  The terms $S_{1,\ell}$ can be bounded similarly, so it follows that $S_1 = o(f(n)^h/n)$.

  For $S_2$, since $f(x) = x^{\frac{1+\kappa}{h}}\phi(x)^{1/h}$, by the definition of $\xi(x)$ we have
  \begin{align}
   S_2 &= \phi(n) \sum_{\substack{(k_1, \ldots, k_{h}) \in \Z_{\geq 1}^{h} \\ b_1k_1 + \cdots + b_{h}k_{h} = n \\ \forall j,\, k_{j}\geq n/\xi(n)}} (k_1\cdots k_h)^{\frac{1+\kappa}{h}-1} \frac{\phi(k_1)^{1/h}}{\phi(n)^{1/h}}\cdots \frac{\phi(k_h)^{1/h}}{\phi(n)^{1/h}} \nonumber \\
   &\sim \phi(n) \sum_{\substack{(k_1, \ldots, k_{h}) \in \Z_{\geq 1}^{h} \\ b_1k_1 + \cdots + b_{h}k_{h} = n \\ \forall j,\, k_{j}\geq n/\xi(n)}} (k_1\cdots k_h)^{\frac{1+\kappa}{h}-1} \label{eqrwst}
  \end{align}
  Using the same methods used to calculate $S_1$, one can show that
  \begin{equation*}
   S_3 := \phi(n)\sum_{\substack{(k_1,\ldots,k_h)\in \Z_{\geq 1}^h \\ b_1k_1+\cdots+b_hk_h = n \\ \exists j \,|\, k_{j}< n/\xi(n)}} (k_1\cdots k_h)^{\frac{1+\kappa}{h}-1} = o\bigg(\frac{f(n)^{h}}{n}\bigg).
  \end{equation*} 
  Thus, since $1/h \leq \frac{1+\kappa}{h} \leq 1$, we may apply Lemma \ref{expgoodnp}, so \eqref{eqrwst} implies that
  \[ S_2 \sim \frac{\Gamma(\frac{1+\kappa}{h})^{h}}{\Gamma(1+\kappa)} \frac{n^{\kappa}\phi(n)}{(b_1\cdots b_h)^{\frac{1+\kappa}{h}}} \]
  concluding the proof.
 \end{proof}
 
 \begin{rem}[Case $\phi \equiv 1$]\label{phic}
  In this case, $F(x) = x^{\kappa}$ for some $\kappa>0$. We have $\E(|\A\cap[1,x]|) = c\int_{1}^{x} t^{\frac{1+\kappa}{h}-1}\,\mathrm{d}t + O_c(1) = c\frac{h}{1+\kappa}x^{(1+\kappa)/h} + O_c(1)$, and $|\A\cap [1,n]|$ is a boolean polynomial of degree $1$. For $d=1$, applying Theorem \ref{kimvu} taking $\lambda = 2\log n$ then yields
  \[ \Pr\big(\big||\A\cap[1,n]| - \E(|\A\cap[1,n]|)\big| \geq 16\log n\, \E(|\A\cap[1,n]|)^{1/2} \big) \ll n^{-2}. \]
  By the Borel--Cantelli lemma, $|\A\cap[1,x]| \stackrel{\textnormal{a.s.}}{=} c\frac{h}{1+\kappa}x^{(1+\kappa)/h} + O(x^{(1+\kappa)/2h}\log x)$.
  
  Furthermore, redoing the calculations at the beginning of Lemma \ref{expgood}, using Lemma \ref{expgoodnp} we obtain
  \begin{equation*}
   \E(r_{\A,h}(n)) = c^{h} \frac{\Gamma(\frac{1+\kappa}{h})^{h}}{\Gamma(1+\kappa)}\, \frac{n^{\kappa}}{(b_1\cdots b_{h})^{\frac{1+\kappa}{h}}} + O_c\big(n^{\frac{(h-1)(1+\kappa)}{h}-1}\big).
  \end{equation*}
 \end{rem}
 
 \subsection{Proof of Theorem \ref{MT1}}
  From \eqref{excnon} and Lemma \ref{fundLem}, we have $\E(r_{\A,h}(n)) = \E(\rho_{\A,h}(n)) + O_c(\frac{f(n)^{h-1}}{n})$, so $\E(r^{*}_{\A,\ell}(k)) \ll k^{\frac{\ell(1+\kappa)}{h}-1 + o(1)}$ ($1\leq \ell \leq h-1$), and by Lemma \ref{expgood}, $\E(\rho_{\A,h}(n)) \sim c^{h} d_{h,\kappa}\, F(n)$ for a certain $d_{h,\kappa}$. So choose $c := d_{h,\kappa}^{-1/h}$.  \medskip
  
  \noindent
  $\bullet~\textnormal{\underline{Case $\kappa>0$}:}$   
  Let $Y = \rho_{\A,h}(n)$ and take $\lambda = (h+1)\log n$ in Theorem \ref{kimvu}. We have\footnote{The $o(1)$ term in the exponent of the error term comes from the fact that, unless we assume $\phi$ increasing, the best general estimate for $\max_{k\leq n}\phi(k)$ is $n^{o(1)}$.}
  \begin{align*}
   \lambda^{h}\bigg(\frac{E'}{E}\bigg)^{1/2} &\ll (\log n)^{h} \Bigg(\frac{1 + \max\limits_{1\leq \ell\leq h-1} \max\limits_{k \leq n} \E(r^{*}_{\A, \ell}(k))}{\E(\rho_{\A, h}(n))}\Bigg)^{1/2} \\
   &\ll \Bigg(\frac{n^{\max\{0, \frac{(h-1)\kappa}{h}-1\} + o(1)}}{n^{\kappa +o(1)}}\Bigg)^{1/2} = n^{-\frac{\kappa}{2} + \max\{0, \frac{(h-1)\kappa}{2h} - \frac{1}{2}\} + o(1)} \qquad (= o(1)).
  \end{align*}
  Thus,
  \begin{equation}
   \Pr\Bigg(\big|\rho_{\A, h}(n) - \E(\rho_{\A, h}(n))\big| \geq \frac{ 8^{h}\sqrt{h!}}{n^{\frac{\kappa}{2} - \max\{0, \frac{(h-1)\kappa}{2h}-\frac{1}{2}\} + o(1)}}\, \E(\rho_{\A, h}(n)) \Bigg) \ll n^{-2} \label{eqerrortm}
  \end{equation}
  which by the Borel--Cantelli lemma implies that $\rho_{\A,h}(n) \stackrel{\text{a.s.}}{\sim} \E(\rho_{\A,h}(n))$.
  
  To bound the non-exact solutions, note that by \eqref{excnon}, $r_{\A,h}(n) - \rho_{\A,h}(n)$ equals the number of exact solutions to a finite number of linear equations of smaller length, with coefficients bounded by $\max_{i} b_i$. So it suffices to show that $\rho_{\A,t}(n):=\rho_{\A,t}^{(c_1,\ldots,c_{t})}(n) \stackrel{\text{a.s.}}{=} O(n^{\max\{0,(1-\frac{1}{h})\kappa - 1\}+o(1)})$ for every equation $c_1k_1+\cdots+c_t k_t$ produced by non-exact solutions to $b_1k_1 +\cdots +b_{h}k_{h}$, as in \eqref{excnon}. Since $t\leq h-1$, for $Y = \rho_{\A,t}(n)$ we have, in the notation of Lemma \ref{kimvu},
  \begin{align*}
   E' &\ll n^{\max\{0,\frac{(h-2)\kappa}{h} - 1\} + o(1)}, \quad E = \max\{E', n^{(1-\frac{1}{h})\kappa - 1 + o(1)}\},
  \end{align*}
  so it follows that
  \[ \Pr\left(|\rho_{\A,t}(n)- \E(\rho_{\A,t}(n))| \geq 8^{h}\sqrt{h!}\, (\log n)^{h}\, n^{\max\{0,(1-\frac{1}{h})\kappa - 1\}+o(1)}\right) \ll n^{-2}, \]
  which by the Borel--Cantelli lemma implies that $\rho_{\A,t}(n) \stackrel{\text{a.s.}}{\ll} n^{\max\{0,(1-\frac{1}{h})\kappa - 1\}+o(1)}$.\medskip
  
  %%%%%%%%%%%%%%%%%%%%%%%%%%%%%%%%%%%%%%%%%%%%%%%%%%%%%%%%%%%%%%%%%%%%%%%%%%%%%%%%%%%%%%%%%%%%%%%%%%%%
  \noindent
  $\bullet~\textnormal{\underline{Case $\kappa=0$}:}$  
  For $\kappa=0$, we apply Theorem \ref{vuORI} to $\rho_{\A,h}^{(\delta\text{-normal})}(n)$ --- after this, it will suffice to show that $\rho^{(\delta\text{-small})}_{\A,h}(n)$ and $r_{\A,h}(n)-\rho_{\A,h}(n)$ are almost surely $O(1)$. The function $\rho_{\A,h}^{(\delta\text{-normal})}(n)$ is a homogeneous, simple boolean polynomial of degree $h$, with partial derivatives bounded by, for $1\leq j\leq h-1$,
  \begin{align*}
    \E_{j}(\rho_{\A,h}^{(\delta\text{-normal})}(n)) \leq \max_{n^{\delta}\leq k\leq n} \E(r^{*}_{\A,h-j}(k)) &\ll_c \max_{n^{\delta}\leq k\leq n} \frac{f(k)^{h-j}}{k} \ll n^{-\delta\frac{j}{h}+o(1)} 
  \end{align*}
  by Lemma \ref{fundLem}. Each monomial of $\rho_{\A,h}(n)$ appears at most $h!$ times, so 
  \[ Y = \frac{1}{h!}\rho^{(\delta\text{-normal})}_{\A,h}(n) \]
  is a normal polynomial. Thus, taking $0<\alpha < \frac{\delta}{h}$, $\beta = 2$, $K = K(\alpha,\beta,h)$ in Theorem \ref{vuORI}, since $\E(\rho^{(\delta\text{-normal})}_{\A,h}(n))/\log n \to \infty$ as $x\to\infty$ (by Lemmas \ref{fundLem}, \ref{expdsml} and our assumptions), we have
  \[ \Pr\left(|\rho^{(\delta\text{-normal})}_{\A,h}(n) - \E(\rho^{(\delta\text{-normal})}_{\A,h}(n))| \geq (h!\lambda\, \E(\rho^{(\delta\text{-normal})}_{\A,h}(n)))^{1/2} \right) \ll n^{-2} \]
  by taking $\lambda = 32h K\log n$. Since $(\lambda \E(\rho^{(\delta\text{-normal})}_{\A,h}(n)))^{1/2} = o(\E(\rho^{(\delta\text{-normal})}_{\A,h}(n)))$, the Borel--Cantelli lemma implies that $\rho^{(\delta\text{-normal})}_{\A,h}(n) \stackrel{\text{a.s.}}{\sim} \E(\rho^{(\delta\text{-normal})}_{\A,h}(n))$.
  
  To bound $\rho^{(\delta\text{-small})}_{\A,h}(n)$, we use Lemma \ref{prodHat}. By Lemma \ref{fundLem}, we have $\E(r^{*}_{\A,\ell}(n))\ll n^{-\frac{1}{h} +o(1)}$ for $1\leq \ell \leq h-1$, and by Lemma \ref{expdsml} we have $\E(\rho^{(\delta\text{-small})}_{\A,h}(n)) \ll n^{-\frac{1-\delta}{2h} + o(1)}$ (since we can take $\vartheta_{f} = \frac{1}{2h}$ in Lemma \ref{orpiCHAR} (ii)). By the disjointness lemma \ref{disjlm},
  \[ \Pr(\widehat{r}^{*}_{\A,\ell}(n) \geq T) \leq \bigg(\frac{e}{T}\bigg)^{T} n^{-T/h + o(1)} \ll n^{-3} \] 
  for large $T\in \R_{>0}$ --- and similarly for $\widehat{\rho}^{(\delta\text{-small})}_{\A,h}(n)$. Therefore, $\widehat{r}^{*}_{\A,\ell}(n) \stackrel{\text{a.s.}}{\ll} 1$ for $1\leq \ell\leq h-1$ and $\widehat{\rho}^{(\delta\text{-small})}_{\A,h}(n) \stackrel{\text{a.s.}}{\ll} 1$ by the Borel--Cantelli lemma. For each $2\leq \ell \leq h-1$ it follows from the union bound that
  \[ \sum_{k\leq n} \Pr(\widehat{r}^{*}_{\A,\ell}(n) \geq T) \ll n^{-2}, \]
  so again by Borel--Cantelli we have $\max_{k\leq n} \widehat{r}^{*}_{\A,\ell}(k) \stackrel{\text{a.s.}}{\ll} 1$. Plugging this into Lemma \ref{prodHat} yields $\rho^{(\delta\text{-small})}_{\A,h}(n) \stackrel{\text{a.s.}}{\ll} \widehat{\rho}^{(\delta\text{-small})}_{\A,h}(n) \stackrel{\text{a.s.}}{\ll} 1$, as desired.
  
  To bound $r_{\A,h}(n)-\rho_{\A,h}(n)$, by \eqref{excnon} it suffices to bound $\rho_{\A,t}(n):=\rho_{\A,t}^{(c_1,\ldots,c_{t})}(n)$ for $t\leq h-1$. If $t=1$ then $\rho_{\A,t}(n) \leq 1$ trivially. For $t\geq 2$, we apply Lemma \ref{prodHat} again, obtaining $\rho_{\A,t}(n) \stackrel{\text{a.s.}}{\ll} \widehat{\rho}_{\A,t}(n) \stackrel{\text{a.s.}}{\ll} 1$, concluding the proof. \hfill$\square$

%%%%%%%%%%%%%%%%%%%%%%%%% 
\subsection{Proof of Corollary \ref{MT12}}
  With Remark \ref{phic}, we can redo the calculation in the case $\kappa >0$ of the proof of Theorem \ref{MT1}, obtaining 
  \[ \lambda^{h}\bigg(\frac{E'}{E}\bigg)^{1/2} \ll n^{-\frac{\kappa}{2} + \max\{0,\, \frac{(h-1)\kappa}{2h} - \frac{1}{2}\}}(\log n)^{h}. \]
  The equivalent to inequality \eqref{eqerrortm} together with the Borel--Cantelli lemma provides the almost sure estimate
  \begin{align*}
   \rho_{\A,h}(n) \stackrel{\text{a.s.}}{=} \E(\rho_{\A, h}(n)) + O\big(n^{\frac{\kappa}{2} + \max\{0, \frac{(h-1)\kappa}{2h}-\frac{1}{2}\}} (\log n)^{h}\big).
  \end{align*}
  Since $\rho_{\A,t}^{(c_1,\ldots,c_{t})}(n) \stackrel{\text{a.s.}}{\ll} n^{\max\{0,(1-\frac{1}{h})\kappa - 1\}+o(1)}$ for $t\leq h-1$, as shown in the proof of Theorem \ref{MT1}, it follows from \eqref{excnon} and
  \[ \frac{\kappa}{2} + \max\bigg\{0,\, \frac{(h-1)\kappa}{2h} - \frac{1}{2} \bigg\} > \max\bigg\{0,\bigg(1-\frac{1}{h}\bigg)\kappa - 1\bigg\} \]
  that
  \begin{align*}
   r_{\A,h}(n) \stackrel{\text{a.s.}}{=} \E(r_{\A, h}(n)) + O\big(n^{\frac{\kappa}{2} + \max\{0, \frac{(h-1)\kappa}{2h}-\frac{1}{2}\}} (\log n)^{h}\big).
  \end{align*}
  The result then follows from Remark \ref{phic}, by simplifying the expression
  \[ O\big(n^{\frac{(h-1)(1+\kappa)}{h}-1}\big) + O\big(n^{\frac{\kappa}{2} + \max\{0, \frac{(h-1)\kappa}{2h}-\frac{1}{2}\}} (\log n)^{h}\big) \]
  depending on $\kappa$ and $h$. \hfill$\square$

%%%%%%%%%%%%%%%%%%%%%%%%%%%%%%%%%%%%%%%%%%%%%%%%%% 
%%%%%%%%%%%%%%%%%%%%%%%%%%%%%%%%%%%%%%%%%%%%%%%%%%
\section{Order of magnitude case: Theorem \ref{MT21}}
 Since Theorem \ref{MT2} contains the case $h=2$ of Theorem \ref{MT21}, assume $h\geq 3$. Let $F$ be a locally integrable increasing function that satisfies $F(2x)\ll F(x)$ and
 \[ \log x \ll F(x) \ll x^{h-1}. \]
 Take $f(x) := (xF(x))^{1/h}$. Note that $f$ satisfies the conditions of Lemma \ref{orpiCHAR}, for $\vartheta = 1/h$. Since $F(x)^{1/h} = f(x)/x^{1/h}$ is increasing, for $1\leq k\leq n$ and $1\leq \ell\leq h-1$ we have
 \[ f(k)^{h-\ell} \leq k^{1-\ell/h} \bigg(\frac{f(n)}{n^{1/h}}\bigg)^{h-\ell}. \]
 Hence, by Lemma \ref{fundLem},
 \begin{align}
  \E(r^{*}_{\A,h-\ell}(k)) \ll_c \frac{f(k)^{h-\ell}}{k} \leq k^{-\ell/h}\, F(n)^{1-\ell/h}. \label{bdINC}
 \end{align}

 We partition $\N=\Z_{\geq 0}$ into
 \begin{equation}
  \begin{aligned}
   \N^{(1)} &:= \{n\in\N ~|~ F(n) \geq (\log n)^{3h^2} \}, \\
   \N^{(2)} &:= \{n\in\N ~|~ F(n) < (\log n)^{3h^2} \}. \\
  \end{aligned}
 \end{equation}
 Our strategy is as follows. On $\N^{(1)}$, expectations are large enough to apply Kim--Vu to $\rho_{\A,\ell}(n)$ and hence to $r_{\A,\ell}(n)$ for $2\leq \ell\leq h$. On $\N^{(2)}$ we first treat $\delta$-normal representations by Vu's theorem \ref{vuORI}, and reduce the remaining contribution to $\delta$-small representations. These are bounded inductively using \eqref{mchdeg} and estimates for $r_{\A,\ell}(n)$ with $\ell\leq h-1$.

 \begin{lem}\label{real-lem1}
  Let $2\leq \ell \leq h-1$, and let $(c_1,\ldots,c_{\ell})\in\Z_{\geq 1}$ with $c_1+\ldots+c_{\ell} \leq \ell \max_{j\leq h} b_j$. Write $r_{\A,\ell}(n) := r_{\A,\ell}^{(c_1,\ldots,c_{\ell})}(n)$. Then we have
  \[ r_{\A,\ell}(n) \stackrel{\textnormal{a.s.}}{\ll} F(n)^{\ell/h} \quad \text{for } n\in \N^{(1)}. \]
 \end{lem}
 \begin{proof}
  It suffices to work with $\rho_{\A,\ell}(n)$. By \eqref{excnon}, the difference $r_{\A,\ell}(n)-\rho_{\A,\ell}(n)$ counts exact solutions to a finite family of equations of smaller length. Hence it suffices to prove that, for each $2\leq t\leq \ell$ and each corresponding coefficient vector $(d_1,\ldots,d_t)$ arising in \eqref{excnon},
  \[ \rho^{(d_1,\ldots,d_t)}_{\A,t}(n) \stackrel{\textnormal{a.s.}}{\ll} F(n)^{t/h} \quad \text{for } n\in\N^{(1)}. \]
  Since there are only finitely many such equations, it is enough to treat a fixed $t$ and a fixed coefficient vector; the argument is identical.

  So let $Y = \rho_{\A,\ell}(n)$ and take $\lambda = (\ell+1)\log n$ in Theorem \ref{kimvu}. By Lemma \ref{fundLem} and \eqref{bdINC} we have\footnote{In fact, one has the sharper bound $\E(\rho_{\A,\ell}(n)) \ll F(n)^{\ell/h}/n^{1-\ell/h}$, but $\E(\rho_{\A,\ell}(n)) \ll F(n)^{\ell/h}$ is sufficient here.}
  \[ \E(\rho_{\A,\ell}(n)) \ll F(n)^{\ell/h}. \]
  Moreover, again by \eqref{bdINC}, for $n\in\N^{(1)}$,
  \[ 1 + \max\limits_{1\leq t\leq \ell-1} \max\limits_{k \leq n} \E(r^{*}_{\A, t}(k)) \ll 1+ F(n)^{(\ell-1)/h} \ll \frac{F(n)^{\ell/h}}{(\log n)^{3h}}. \]
  In the notation of Theorem \ref{kimvu}, this gives $E \ll F(n)^{\ell/h}$ and $E' \ll F(n)^{\ell/h}/(\log n)^{3h}$. Hence
  \[ \lambda^{\ell} (E'E)^{1/2} \ll \frac{\lambda^{\ell}}{(\log n)^{3h/2}} F(n)^{\ell/h} = o(F(n)^{\ell/h}). \]
  Therefore Theorem \ref{kimvu} yields
  \[ \Pr(|\rho_{\A, \ell}(n) - \E(\rho_{\A, \ell}(n))| \geq F(n)^{\ell/h}) \ll n^{\ell-1}e^{-\lambda} = n^{-2}. \]
  By the Borel--Cantelli lemma, $\rho_{\A, \ell}(n) \stackrel{\textnormal{a.s.}}{\ll} F(n)^{\ell/h}$ for $n\in\N^{(1)}$, which concludes the proof.
 \end{proof}

 \begin{lem}\label{rN1}
  $r_{\A,h}(n) \stackrel{\textnormal{a.s.}}{\asymp} F(n)$ for $n\in \N^{(1)}$.
 \end{lem}
 \begin{proof}
  Let $Y = \rho_{\A,h}(n)$ and take $\lambda = (h+1)\log n$ in Theorem \ref{kimvu}. By \eqref{bdINC}, for $n\in \N^{(1)}$ and $k\leq n$ we have
  \[ \E(r^{*}_{\A, \ell}(k)) \ll F(n)^{1-1/h} \leq \frac{F(n)}{(\log n)^{3h}} , \]
  and so, by Lemma \ref{fundLem},
  \begin{align*}
   \lambda^h \bigg(\frac{E'}{E}\bigg)^{1/2}
   &\ll (\log n)^{h} \Bigg(\frac{1 + \max\limits_{1\leq \ell\leq h-1} \max\limits_{k \leq n} \E(r^{*}_{\A, \ell}(k))}{\E(\rho_{\A, h}(n))}\Bigg)^{1/2} \\
   &\ll_c (\log n)^{h} \Bigg(\frac{\frac{1}{(\log n)^{3h}} F(n)}{F(n)}\Bigg)^{1/2} = o(1).
  \end{align*}
  Therefore Theorem \ref{kimvu} yields
  \[ \Pr\left(|\rho_{\A, h}(n) - \E(\rho_{\A, h}(n))| \geq \tfrac{1}{2}\, \E(\rho_{\A, h}(n)) \right) \ll n^{-2}. \]
  By the Borel--Cantelli lemma, it follows that $\rho_{\A, h}(n) \stackrel{\textnormal{a.s.}}{\asymp} F(n)$ for $n\in\N^{(1)}$.

  To pass from $\rho_{\A,h}$ to $r_{\A,h}$, we argue as in Lemma \ref{real-lem1}. By \eqref{excnon}, the difference $r_{\A,h}(n)-\rho_{\A,h}(n)$ counts exact solutions to finitely many equations of smaller length. Each such contribution is almost surely $\ll F(n)$ by Lemma \ref{real-lem1} (with $\ell\leq h-1$), and hence $r_{\A,h}(n) \stackrel{\textnormal{a.s.}}{\asymp} F(n)$ for $n\in\N^{(1)}$.
 \end{proof}

 We now work on $\N^{(2)}$. We start by focusing on $\delta$-normal representations.

 \begin{lem}\label{real-lem2}
  In the notation of Lemma \ref{real-lem1}, for $2\leq \ell \leq h-1$ we have
  \[ r^{(\delta\textnormal{-normal})}_{\A,\ell}(n) \stackrel{\textnormal{a.s.}}{\ll} 1 \quad \text{for } n\in \N^{(2)}. \]
 \end{lem}
 \begin{proof}
  As in Lemma \ref{real-lem1}, it suffices to work with $\rho^{(\delta\textnormal{-normal})}_{\A,\ell}(n)$. Indeed, by \eqref{excnon}, to bound $r^{(\delta\textnormal{-normal})}_{\A,\ell}(n)$ it is enough to show that $\rho^{(\delta\textnormal{-normal})}_{\A,t}(n) = \rho^{(\delta\textnormal{-normal})(d_1,\ldots,d_t)}_{\A,t}(n) \stackrel{\textnormal{a.s.}}{\ll} 1$ for every $t\leq \ell$.

  Since the solutions counted by $\rho^{(\delta\textnormal{-normal})}_{\A,\ell}$ only use elements $\geq n^{\delta}$, we may apply a suitable modification of Lemma \ref{prodHat}, obtaining\footnote{This follows by replacing \eqref{mchdeg} by
  \[ r^{(\delta\textnormal{-normal})}_{\A,\ell}(n) \leq \sum_{j=1}^{\ell} \sum_{\substack{k \in S \\ S\in \mathcal{M}}} r^{*}_{\A\cap[n^{\delta},n],\ell-1}(n-b_j k) \leq \ell\cdot \ell! \, \widehat{r}^{(\delta\textnormal{-normal})}_{\A,\ell}(n) \left(\max_{k\leq n} r^{*}_{\A\cap [n^{\delta},n],\ell-1}(k) \right). \]}
  \begin{equation}
   \rho^{(\delta\textnormal{-normal})}_{\A,\ell}(n) \ll \widehat{\rho}^{(\delta\textnormal{-normal})}_{\A,\ell}(n) \left(\max_{n^{\delta}\leq k\leq n} \widehat{r}^{*}_{\A,\ell-1}(k) \right) \cdots \left(\max_{n^{\delta} \leq k\leq n} \widehat{r}^{*}_{\A,2}(k) \right). \label{rhonorm}
  \end{equation}
  Moreover, by \eqref{bdINC} we have $\E(\rho^{(\delta\textnormal{-normal})}_{\A,\ell}(n)) \leq n^{-1/h + o(1)}$ for $n\in\N^{(2)}$, and for $1\leq t\leq \ell-1$ and $n^{\delta} \leq k\leq n$ we have $\E(r^{*}_{\A,\ell-t}(k)) \leq n^{-\delta/h + o(1)}$. Therefore, by the disjointness lemma \ref{disjlm}, for $T>0$ large enough,
  \[ \Pr(\widehat{\rho}^{(\delta\textnormal{-normal})}_{\A,\ell}(n) \geq T) \leq \bigg(\frac{e}{T}\bigg)^T n^{-\frac{1}{h} T + o(1)} \leq n^{-2+o(1)}, \]
  and similarly,
  \begin{align*}
   \Pr\Big(\max_{n^{\delta}\leq k \leq n} \widehat{r}^{*}_{\A,\ell-t}(k) \geq T\Big)
   &\leq \sum_{n^{\delta}\leq k\leq n} \Pr(\widehat{r}^{*}_{\A,\ell-t}(k) \geq T) \\
   &\leq \sum_{n^{\delta}\leq k\leq n} \bigg(\frac{e}{T}\bigg)^T n^{-\delta\frac{1}{h} T + o(1)}
   \leq n^{-2+o(1)}.
  \end{align*}
  By the Borel--Cantelli lemma, all factors on the right-hand side of \eqref{rhonorm} are almost surely bounded, and hence $\rho^{(\delta\textnormal{-normal})}_{\A,\ell}(n) \stackrel{\textnormal{a.s.}}{\ll} 1$ for $n\in\N^{(2)}$.
 \end{proof}

 \begin{lem}\label{rdltN2}
  For $c>1$ sufficiently large, $r^{(\delta\textnormal{-normal})}_{\A,h}(n) \stackrel{\textnormal{a.s.}}{\asymp} F(n)$ for $n\in \N^{(2)}$.
 \end{lem}
 \begin{proof}
  In the notation of Theorem \ref{vuORI}, for $n\in\N^{(2)}$ we have by \eqref{bdINC} that for every $1\leq j\leq h-1$,
  \[ \E_{j}(\rho^{(\delta\textnormal{-normal})}_{\A,h}(n)) \leq \max_{n^{\delta}\leq k\leq n} \E(r^{*}_{\A,h-j}(k)) \leq n^{-\delta j/h + o(1)}. \]
  We apply Theorem \ref{vuORI} with $Y=\frac{1}{h!}\rho^{(\delta\textnormal{-normal})}_{\A, h}(n)$ and $\beta = 2$. By Lemmas \ref{fundLem}, \ref{expdsml}, we can choose $c>1$ large enough so that $\E(\frac{1}{h!}\rho^{(\delta\textnormal{-normal})}_{\A, h}(n)) \gg \min\{cf(n),n\}^{h}/n \geq 32 h K \log n$. Taking $\lambda = 32h K\log n$, we obtain
  \[ \Pr\left(|\rho^{(\delta\textnormal{-normal})}_{\A,h}(n) - \E(\rho^{(\delta\textnormal{-normal})}_{\A,h}(n))| \geq (h!\lambda\, \E(\rho^{(\delta\textnormal{-normal})}_{\A,h}(n)))^{1/2} \right) \ll n^{-2}. \]
  By the Borel--Cantelli lemma, it follows that $\rho^{(\delta\textnormal{-normal})}_{\A, h}(n) \stackrel{\textnormal{a.s.}}{\asymp} F(n)$ for $n\in \N^{(2)}$.

  As in Lemma \ref{rN1}, to pass from $\rho^{(\delta\textnormal{-normal})}_{\A,h}$ to $r^{(\delta\textnormal{-normal})}_{\A,h}$ it suffices to bound the contributions of exact solutions of smaller length. This is exactly Lemma \ref{real-lem2} (applied with $\ell\leq h-1$), and the proof is complete.
 \end{proof}

 All that is left to show is that $r^{(\delta\textnormal{-small})}_{\A,h}(n) \stackrel{\textnormal{a.s.}}{\ll} F(n)$ in $\N^{(2)}$. This is a consequence of the following lemma:

 \begin{lem}\label{real-lem3}
  In the notation of Lemma \ref{real-lem1}, for $2\leq \ell \leq h$ we have
  \[ r^{(\delta\textnormal{-small})}_{\A,\ell}(n) \stackrel{\textnormal{a.s.}}{\ll} F(n)^{(\ell-1)/h} \quad \text{for } n\in \N^{(2)}. \]
 \end{lem}
 \begin{proof}
  By \eqref{mchdeg}, we have
  \[ r^{*(\delta\textnormal{-small})}_{\A,\ell}(n) \ll \widehat{r}^{*(\delta\textnormal{-small})}_{\A,\ell}(n) \left(\max_{k\leq n} r^{*}_{\A,\ell-1}(k) \right). \]
  By Lemma \ref{expdsml} and \eqref{bdINC}, and since $F$ is increasing (so we may take $\vartheta = \vartheta_f = 1/h$ in Lemma \ref{orpiCHAR} (ii)), we have
  \[ \E(r^{*(\delta\textnormal{-small})}_{\A,\ell}(n)) \ll n^{-(1-\delta)\vartheta} F(n)^{\ell/h} = n^{-(1-\delta)/h + o(1)} \qquad (\text{for } n\in\N^{(2)}). \]
  Hence, by the disjointness lemma \ref{disjlm}, for large $T>0$,
  \[ \Pr(\widehat{r}^{*(\delta\textnormal{-small})}_{\A,\ell}(n) \geq T) \leq \bigg(\frac{e}{T}\bigg)^T n^{-(1-\delta) T/h + o(1)} \leq n^{-2+o(1)}, \]
  so by the Borel--Cantelli lemma, $\widehat{r}^{*(\delta\textnormal{-small})}_{\A,\ell}(n) \stackrel{\textnormal{a.s.}}{\ll} 1$. Therefore,
  \begin{equation}
   r^{*(\delta\textnormal{-small})}_{\A,\ell}(n) \stackrel{\textnormal{a.s.}}{\ll} \max_{k\leq n} r^{*}_{\A,\ell-1}(k) \quad \text{for } n\in \N^{(2)}.\label{bdbymax}
  \end{equation}

  We prove the lemma by induction on $\ell$. For $\ell=2$, it follows from \eqref{bdbymax} that $r^{*(\delta\textnormal{-small})}_{\A,2}(n) \stackrel{\textnormal{a.s.}}{\ll} \max_{k\leq n} \mathbbm{1}_{\A}(k) \leq 1$ for $n\in \N^{(2)}$, which is $\ll F(n)^{1/h}$ since $F(n) \gg \log n$.

  Suppose now that the claim holds for some $2\leq \ell \leq h-1$, so $r^{*(\delta\textnormal{-small})}_{\A,\ell}(n) \stackrel{\textnormal{a.s.}}{\ll} F(n)^{(\ell-1)/h}$ for $n\in\N^{(2)}$. For $n\in\N^{(2)}$, every representation counted by $r^{*}_{\A,\ell}(n)$ is either $\delta$-normal or $\delta$-small, hence
  \[ r^{*}_{\A,\ell}(n) \leq r^{*(\delta\textnormal{-normal})}_{\A,\ell}(n) + r^{*(\delta\textnormal{-small})}_{\A,\ell}(n) \stackrel{\textnormal{a.s.}}{\ll} 1 + F(n)^{(\ell-1)/h} \stackrel{\textnormal{a.s.}}{\ll} F(n)^{\ell/h}. \]
  For $n\in\N^{(1)}$ we have $r^{*}_{\A,\ell}(n) \stackrel{\textnormal{a.s.}}{\ll} F(n)^{\ell/h}$ by Lemma \ref{real-lem1}. Consequently,
  \[ \max_{k\leq n} r^{*}_{\A,\ell}(k) \stackrel{\textnormal{a.s.}}{\ll} \max_{k\leq n} F(k)^{\ell/h} \leq F(n)^{\ell/h}, \]
  since $F$ is increasing. Plugging this into \eqref{bdbymax} gives $r^{*(\delta\textnormal{-small})}_{\A,\ell+1}(n) \stackrel{\textnormal{a.s.}}{\ll} F(n)^{\ell/h}$ for $n\in\N^{(2)}$, completing the induction.
 \end{proof}

 Putting Lemmas \ref{rN1}, \ref{rdltN2}, and the case $\ell=h$ of Lemma \ref{real-lem3} together yields Theorem \ref{MT21}.

 %%%%%%%%%%%%%%%%%%%%%%%%% 
\section{Order of magnitude case: Theorem \ref{MT2}}
 The case $h=2$ was already considered in Subsection \ref{caseh2}, so suppose $h\geq 3$. Let $f$ be a positive locally integrable real function satisfying $\int_{1}^{x} \frac{f(t)}{t}\,\mathrm{d}t \asymp f(x)$ in the range
 \begin{equation}
  (n\,\psi(n))^{1/h} \ll f(n) \ll (n\,\psi(n))^{1/(h-1)}, \label{rangenidl}
 \end{equation}
 where $\psi(x) \gg \log x$ is some increasing slowly varying function, and select some 
 \begin{equation}
   1-\frac{1}{4h(1-(h-1)\min\{\vartheta_f,\frac{1}{h}\})} < \delta < 1, \label{dltcnd}
 \end{equation}
 where $\vartheta_f$ is as in Lemma \ref{orpiCHAR} (ii).
 
 \begin{lem}\label{collage}
  For $\delta>0$ as in \eqref{dltcnd}, and for $c>1$ sufficiently large, we have
  \[ \rho^{(\delta\textnormal{-normal})}_{\A,h}(n) \stackrel{\textnormal{a.s.}}{\asymp} \dfrac{f(n)^{h}}{n}. \]
 \end{lem}
 \begin{proof}
  Partition $\N = \Z_{\geq 0}$ into
  \begin{align*}
   \N^{(1)} := \bigg\{ n\in \N ~\bigg|~ \frac{f(n)^h}{n} \leq n^{1/2h}\bigg\}, \qquad \N^{(2)} := \bigg\{ n\in \N ~\bigg|~ \frac{f(n)^h}{n} > n^{1/2h} \bigg\}.
  \end{align*}
  By Lemmas \ref{fundLem}, \ref{expdsml}, we have 
  \[ \frac{\min\{cf(n),n\}^{h}}{n} \ll \E(\rho^{(\delta\textnormal{-normal})}_{\A,h}(n)) \ll c^{h} \frac{f(n)^{h}}{n}. \]
  Using \eqref{dltcnd}, we can bound the derivatives of $\rho^{(\delta\textnormal{-normal})}_{\A, h}(n)$. Let $\vartheta^{*} := \min\{\vartheta,\frac{1}{h}\}$, where $\vartheta = \vartheta_f$ is as in Lemma \ref{expdsml} (ii). The function $f(x)/x^{\vartheta^{*}}$ is almost increasing, meaning $f(k)/k^{\vartheta^{*}} \ll f(n)/n^{\vartheta^{*}}$ for $n^{\delta}\leq k\leq n$. Rearranging, this implies $f(k) \ll (\frac{n}{k})^{-\vartheta^{*}} f(n)$, and further $\frac{f(k)^{h-1}}{k} \ll (\frac{n}{k})^{1-(h-1)\vartheta^{*}} \frac{f(n)^{h-1}}{n}$. Since $1-(h-1)\vartheta^{*} > 0$, and $k$ is minimized at $n^{\delta}$, we obtain from \eqref{dltcnd}
  \begin{align}
    \max_{n^{\delta}\leq k\leq n} \E(r^{*(\delta\textnormal{-normal})}_{\A, \ell}(k)) \ll_c \max_{n^{\delta}\leq k\leq n} \frac{f(k)^{h-1}}{k} &\ll n^{(1-\delta)(1-(h-1)\vartheta^{*})}\, \frac{f(n)^{h-1}}{n} \nonumber \\
    &< n^{1/4h} \frac{f(n)^{h-1}}{n} \label{eesrt1}
  \end{align}
  for $1\leq \ell\leq h-1$. We study the partitions $\N^{(1)}$ and $\N^{(2)}$ separately.\medskip
  
  \noindent
  $\bullet~\underline{\N^{(1)}}:$
  For $n\in \N^{(1)}$, by \eqref{rangenidl} and \eqref{eesrt1} we have
  \begin{align*}
   \max_{n^{\delta}\leq k \leq n} \E(r^{*(\delta\textnormal{-normal})}_{\A, \ell}(k)) &\ll n^{1/4h}\, \frac{f(n)^{h}}{n} \frac{1}{f(n)} \\
   &< n^{\frac{1}{4h} + \frac{1}{2h} - \frac{1}{h} + o(1)} = n^{-1/4h +o(1)}.
  \end{align*}
  Since the expected values of the derivatives of $Y= \frac{1}{h!}\rho^{(\delta\textnormal{-normal})}_{\A, h}(n)$ are all bounded from above by $O_c(n^{-\alpha})$ in $\N^{(1)}$ for some $\alpha>0$, we apply Theorem \ref{vuORI} with $\beta = 2$. In the notation of Theorem \ref{vuORI}, for $c>1$ large enough we can guarantee that $\E(\frac{1}{h!}\rho^{(\delta\textnormal{-normal})}_{\A, h}(n)) \gg \min\{cf(n),n\}^{h}/n \geq 32 h K \log n$, therefore
  \[ \Pr\left(|\rho^{(\delta\text{-normal})}_{\A,h}(n) - \E(\rho^{(\delta\text{-normal})}_{\A,h}(n))| \geq (h!\lambda\, \E(\rho^{(\delta\text{-normal})}_{\A,h}(n)))^{1/2} \right) \ll n^{-2} \]
  taking $\lambda = 32h K\log n$. By the Borel--Cantelli lemma, $\rho^{(\delta\textnormal{-normal})}_{\A, h}(n) \stackrel{\textnormal{a.s.}}{\asymp} \frac{f(n)^h}{n}$ in $\N^{(1)}$.\medskip
  
  \noindent
  $\bullet~\underline{\N^{(2)}}:$
  Taking $Y = \rho^{(\delta\textnormal{-normal})}_{\A, h}(n)$ and $\lambda = (h+1)\log n$ in Theorem \ref{kimvu}, by Lemma \ref{fundLem} and \eqref{eesrt1} we have
  \[ (\log n)^{h} \Bigg(\frac{1 + \max\limits_{1\leq \ell\leq h-1} \max\limits_{n^{\delta} \leq k \leq n} \E(r^{*}_{\A, \ell}(k))}{\E(\rho^{(\delta\textnormal{-normal})}_{\A, h}(n))}\Bigg)^{1/2} \ll_c n^{-\alpha} \qquad \text{ for }n\in \N^{(2)} \] 
  for some $\alpha>0$. Thus,
  \[ \Pr\bigg(\big|\rho^{(\delta\textnormal{-normal})}_{\A, h}(n) - \E(\rho^{(\delta\textnormal{-normal})}_{\A, h}(n))\big| \geq Dn^{-\alpha} \E(\rho^{(\delta\textnormal{-normal})}_{\A, h}(n)) \bigg) \leq n^{-2} \]
  for some $D = D_{c}$, which by the Borel--Cantelli lemma implies the lemma.
 \end{proof}
 
 \begin{proof}[Proof of Theorem \ref{MT2}]
  It remains to bound $\rho_{\A,h}^{(\delta\textnormal{-small})}(n)$ and, by \eqref{excnon}, $\rho_{\A,t}^{(c_1,\ldots,c_t)}(n)$ for $t\leq h-1$. By the range of $f$ \eqref{rangenidl} and Lemma \ref{fundLem}, we have $\E(r^{*}_{\A,\ell}(k)) \ll_c k^{-\alpha}$ for some $\alpha>0$ for all $1\leq \ell \leq h-2$, and $\E(r^{*}_{\A,h-1}(k)) \ll_c \psi(k)$. By the disjointness lemma \ref{disjlm}, we have
  \[ \Pr(\widehat{r}^{*}_{\A,\ell}(k) \geq T) \leq \bigg(\frac{e}{T}\bigg)^T k^{-\alpha T} \leq k^{-2} \qquad (1\leq \ell \leq h-2),  \]
  and, since $\psi(x) \gg \log x$,
  \begin{align*}
   \Pr\big(\widehat{r}^{*}_{\A,h-1}(k) \geq T \psi(k) \big) &\leq \bigg(\frac{e}{T}\bigg)^{-T\psi(k)} \ll k^{-2},
  \end{align*}
  both for large real $T = T_c>0$. By Lemma \ref{expdsml}, choosing $T$ large enough yields
  \begin{align*}
   \Pr\bigg(\widehat{\rho}^{(\delta\textnormal{-small})}_{\A,h}(n) \geq c^{h-1}\frac{T}{\psi(n)} \frac{f(n)^h}{n} \bigg) &\leq \bigg(\frac{e \,\E(r^{(\delta\textnormal{-small})}_{\A,h}(n))}{c^{h-1}\frac{T}{\psi(n)}\frac{f(n)^h}{n}}\bigg)^{c^{h-1}\frac{T}{\psi(n)}\frac{f(n)^h}{n}} \\
   &\leq (n^{-(1-\delta)\vartheta + o(1)})^{c^{h-1}\frac{T}{\psi(n)}\frac{f(n)^h}{n}} \\
   &\leq n^{-2+o(1)}.
  \end{align*}
  Thus, by the Borel--Cantelli lemma, we have for $k\leq n$:
  \begin{equation*}
   \begin{gathered}
    \widehat{r}^{*}_{\A,\ell}(k) \stackrel{\textnormal{a.s.}}{\ll} 1 \quad (1\leq \ell \leq h-2),\qquad \widehat{r}^{*}_{\A,h-1}(k) \stackrel{\textnormal{a.s.}}{\ll} \psi(k) \leq \psi(n),\\
    \widehat{\rho}^{(\delta\text{-small})}_{\A,h}(n) \stackrel{\textnormal{a.s.}}{\ll} \frac{1}{\psi(n)}\frac{f(n)^{h}}{n}. 
   \end{gathered}%\label{bcqs}
  \end{equation*}
  By Lemma \ref{prodHat}, it follows that $\rho^{(\delta\text{-small})}_{\A,h}(n) \stackrel{\textnormal{a.s.}}{\ll} \frac{f(n)^{h}}{n}$ and $\rho^{(c_1,\ldots,c_t)}_{\A,t}(n) \stackrel{\textnormal{a.s.}}{\ll} \psi(n)$.
 \end{proof}

%%%%%%%%%%%%%%%%%%%%%%%%%%%%%%
\section{What if \texorpdfstring{$F(x) \ll \log x$}{F(x) << log x}?}\label{secwi}
 Let $0< \eps < 1$, and define
 \[ \Pr(n\in \A) := \min\bigg\{c\frac{f(n)}{n},\,1\bigg\} \]
 for each $n\in \Z_{\geq 1}$, where $f(x) = (x \log x)^{1/h}$ and
 \[ c := (1-\eps)^{1/h}\frac{(b_1\cdots b_h)^{1/h^2}}{\Gamma(1/h)}, \]
 so that, by \eqref{slln1}, $|\A \cap [1,x]| \sim hc (x\log x)^{1/h}$. By Lemma \ref{expgood}, we have
 \begin{equation}
  \E(r_{\A,h}(n)) \sim c^{h} \Gamma(1/h)^{h} \frac{\log n}{(b_1\cdots b_h)^{1/h}} \leq (1-\tfrac{\eps}{2}) \log n \label{one8}
 \end{equation}
 for large $n$. To show that $\{r_{\A,h}(n)=0\}$ has high probability, the following inequality will be used:
 
 \begin{lem}[Correlation inequality]\label{correq}
  Let $\Omega$ be a finite set, and $\mathcal{R}$ be a random subset where the events $\{\omega_1\in\mathcal{R}\}$, $\{\omega_2\in\mathcal{R}\}$ are independent for every $\omega_1\neq \omega_2\in\Omega$. Let $S_1$, $\ldots$, $S_n \subseteq \Omega$ be distinct subsets, and suppose that the events $E_i =\{S_i\subseteq \mathcal{R}\}$ satisfy $\Pr(E_i)\leq 1/2$. Then:
  \[ \prod_{i=1}^{n} \Pr( \overline{E}_i) \leq \Pr\bigg( \bigwedge_{i=1}^n \overline{E}_i \bigg) \leq \bigg(\prod_{i=1}^{n} \Pr( \overline{E}_i)\bigg) e^{2\Delta}, \]
  where
  \[ \Delta = \sum_{\substack{1\leq i < j \leq n \\ E_i\cap E_j \neq \varnothing}} \Pr(E_i\wedge E_j) \]
  and $\overline{E}_i$ is the complement of $E_i$.
 \end{lem}
 \begin{proof}
  Boppana--Spencer \cite{bopspe89}.
 \end{proof}
 
 \begin{lem}\label{final2}
  $\displaystyle \sum_{n\geq 1} \Pr(r_{\A,h}(n) = 0) = \infty$.
 \end{lem}
 \begin{proof}
  Let 
  \begin{equation}
   \mathcal{S}[n] = \big\{(k_1,\ldots,k_h)\in \Z_{\geq 0}^h ~|~ b_1 k_1+\cdots+b_h k_h= n \big\}. \label{defS}
  \end{equation}
  Since for every solution $R=(k_1,\ldots,k_h)$ in $\mathcal{S}[n]$ there is some $j$ for which $k_{j} \geq n/h(\max_i b_i)$, for large $n$ we have $\Pr(R) \leq \eps/2$, and hence 
  \[ 1-\Pr(R) \geq e^{-\frac{\Pr(R)}{1-\Pr(R)}} \geq e^{-\frac{2}{2-\eps}\Pr(R)}. \]
  It follows from Lemma \ref{correq} and \eqref{one8} that, for large $n$,
  \begin{align*}
   \Pr(r_{\A,h}(n) = 0) = \Pr\Bigg(\bigwedge_{R\in\mathcal{S}[n]} \overline{R} \Bigg) &\geq \prod_{R\in \mathcal{S}[n]} (1-\Pr(R)) \\
   &\geq e^{-\frac{2}{2-\eps}\sum_{R\in\mathcal{S}[n]}\Pr(R)} \\
   &= e^{-\frac{2}{2-\eps}\,\E(r_{\A,h}(n))} \geq \frac{1}{n},
  \end{align*}
  so $\sum_{n\geq 1} \Pr(r_{\A,h}(n)=0)$ diverges.
 \end{proof} 
 
 To prove Theorem \ref{MT3}, we will use Lemma \ref{final2} together with the fact that the random variables $r_{\A,h}(n)$ ($n\geq 1$) have low correlation. To this end, we apply a generalization of the second Borel--Cantelli lemma due to Kochen and Stone:
 
 \begin{lem}[Kochen--Stone]\label{kcst}
  Let $\{E_n\}_{n\geq 1}$ be a family of events in some probability space, and suppose that $\sum_{n\geq 1} \Pr(E_n) = \infty$. Then,
  \[ \Pr(E_n, \textnormal{infinitely often}) \geq \limsup_{N\to\infty} \frac{\sum_{1\leq n,m \leq N} \Pr(E_n)\Pr(E_{m})}{\sum_{1\leq n , m \leq N} \Pr(E_n \wedge E_m)}.  \]
 \end{lem}
 \begin{proof}
  Cf. Lemma 2 of Yan \cite{yan06}.
 \end{proof}

 In the notation of \eqref{defS}, our events of interest $E_n$ are of the form
 \[ \{r_{\A,h}(n) = 0\} = \bigwedge_{R\in \mathcal{S}[n]} \overline{R}. \]
 By Lemma \ref{correq}, we have, for any $m>n\geq 1$,
 \begin{equation}
  \Pr(r_{\A,h}(n) =0) \geq \prod_{R\in \mathcal{S}[n]}\Pr(\overline{R}) \label{cqand0}
 \end{equation}
 and
 \begin{equation}
  \Pr(r_{\A,h}(n) =0 \text{ and } r_{\A,h}(m) = 0) \leq \bigg(\prod_{R\in \mathcal{S}[n]}\Pr(\overline{R})\bigg) \bigg(\prod_{S\in \mathcal{S}[m]} \Pr(\overline{S})\bigg) e^{\Delta(n,m)}, \label{cqand} 
 \end{equation}
 where
 \begin{equation*}
  \Delta(n,m) := \Bigg(\sum_{\substack{R,S\in \mathcal{S}[n] \\ R \cap S \neq \varnothing}} +\ 2\sum_{\substack{R\in \mathcal{S}[n],\, S\in \mathcal{S}[m] \\ R \cap S \neq \varnothing}} + \sum_{\substack{R,S\in \mathcal{S}[m] \\ R \cap S \neq \varnothing}}\Bigg) \Pr(R\wedge S). %\label{deltamn}
 \end{equation*}

 \begin{lem}\label{Deltacalc}
  $\displaystyle \Delta(n,m) \ll \bigg(1+ \frac{1}{(m-n)^{\frac{1}{h}+o(1)}}\bigg)\frac{1}{n^{\frac{1}{h}+o(1)}} + \frac{1}{m^{\frac{1}{h}+o(1)}}$ 
 \end{lem}
 \begin{proof}
  We start by estimating the sum over $R,S\in\mathcal{S}[m]$ (the $\mathcal{S}[n]$ case is analogous). If $R\cap S =: I \neq \varnothing$, we have $\Pr(R\wedge S) = \Pr(I) \Pr(R\setminus I) \Pr(S\setminus I)$. Thus, from Lemma \ref{fundLem}, as $f(n) = (n\log n)^{1/h}$, we have 
  \begin{align*}
   \sum_{\substack{R,S\in \mathcal{S}[m] \\ R \cap S \neq \varnothing}} \Pr(R\wedge S) &= \sum_{\ell=1}^{h-1} \sum_{\substack{I\subseteq \{0,\ldots,m\}\\ |I|= \ell}} \Pr(I)\Bigg(\sum_{\substack{R,S\in \mathcal{S}[m] \\ R\cap S = I}} \Pr(R\setminus I)\Pr(S\setminus I)\Bigg) \\
   &\leq \sum_{\ell=1}^{h-1} \sum_{k\leq m} \E(r_{\A,\ell}^{*}(k)) \E(r_{\A,h-\ell}^{*}(m-k))^2 \\
   &\ll \sum_{\ell=1}^{h-1} \sum_{k\leq m} \frac{f(k)^{\ell}}{k}\bigg(\frac{f(m-k)^{h-\ell}}{m-k} \bigg)^2 \\
   &\ll \sum_{k\leq m} \frac{f(k)}{k} \bigg(\frac{f(m-k)^{h-1}}{m-k} \bigg)^2 \ll m^{-\frac{1}{h} + o(1)}.
  \end{align*}

  Similarly, for $R\in \mathcal{S}[n]$, $S\in \mathcal{S}[m]$ with $m>n$, we have
  \begin{align*}
   \sum_{\substack{R\in \mathcal{S}[n],\,S\in \mathcal{S}[m] \\ R \cap S \neq \varnothing}} \Pr(R\wedge S) &= \sum_{\ell=1}^{h-1} \sum_{\substack{I\subseteq \{0,\ldots,m\}\\ |I|= \ell}} \Pr(I)\Bigg(\sum_{\substack{R\in\mathcal{S}[n],\, S\in \mathcal{S}[m] \\ R\cap S = I}} \Pr(R\setminus I)\Pr(S\setminus I)\Bigg) \\
   &\ll \sum_{k\leq n} \frac{f(k)}{k} \bigg(\frac{f(n-k)^{h-1}}{n-k} \bigg) \bigg(\frac{f(m-k)^{h-1}}{m-k} \bigg) \\
   &\ll (m-n)^{-\frac{1}{h}+o(1)} \sum_{k\leq n} \frac{f(k)}{k} \bigg(\frac{f(n-k)^{h-1}}{n-k} \bigg) \\
   &\ll (n(m-n))^{-\frac{1}{h} + o(1)},
  \end{align*}
  completing the proof.
 \end{proof}
 
 By Lemma \ref{Deltacalc}, for every $\delta>0$ there is $K\geq 1$ such that $\Delta(n,m)<\delta$ if $m>n\geq K$. Write $E_n := \{r_{\A,h}(n) = 0\}$. By \eqref{cqand0}, \eqref{cqand} we have
 \begin{align*}
  \sum_{\substack{1\leq n<m\leq N \\ n\leq K}} \Pr(E_n\wedge E_m) &\leq e^{\max \Delta(m,n)} \Bigg(\sum_{n=1}^{K} \prod_{R\in \mathcal{S}[n]}\Pr(\overline{R})\Bigg) \Bigg(\sum_{m\leq N} \prod_{S\in \mathcal{S}[m]}\Pr(\overline{S}) \Bigg) \\
  &\ll K \sum_{m\leq N} \Pr(E_m) = o\bigg(\bigg(\sum_{m\leq N} \Pr(E_m)\bigg)^2\bigg).
 \end{align*}
 Thus, we have by Lemma \ref{kcst} that
 \begin{align*}
  \Pr(E_n\text{ infinitely often}) &\geq \limsup_{N\to\infty} \frac{\sum_{K\leq n<m \leq N} \Pr(E_n)\Pr(E_m)}{\sum_{K\leq n < m \leq N} \Pr(E_n \wedge E_m)} \\
  &\geq \limsup_{N\to\infty} \frac{\displaystyle\sum_{K\leq n<m \leq N} \bigg(\prod_{R\in \mathcal{S}[n]}\Pr(\overline{R})\bigg)\bigg(\prod_{S\in \mathcal{S}[m]}\Pr(\overline{S})\bigg)}{\displaystyle\sum_{K\leq n < m \leq N}\bigg(\prod_{R\in \mathcal{S}[n]}\Pr(\overline{R})\bigg)\bigg(\prod_{S\in \mathcal{S}[m]}\Pr(\overline{S})\bigg) e^{\Delta(n,m)}} \\
  &\geq e^{-\delta}
 \end{align*}
 for every $\delta > 0$. This concludes the proof of Theorem \ref{MT3}. 
 
%  \begin{xrem}
%   For a more precise calculation of the density of $\{n ~|~ r_{\A,h}(n)=0\}$ in the case $f(x) = x^{1/h}$, cf. Landreau \cite{lan95}.
%  \end{xrem}
%  
 
%%%%%%%%%%%%%%%%%%%%%%%
\addtocontents{toc}{\protect\setcounter{tocdepth}{0}}
\section*{Acknowledgements}
 I would like to thank Andrew Granville for his guidance and advice during the writing of this article. I also thank Cihan Sabuncu and Tony Haddad for many fruitful discussions.
 
\addtocontents{toc}{\protect\setcounter{tocdepth}{1}}
%%%%%%%%%%%%%%%%%%%%%%%

% ----------------------------------------------------------------
\bibliographystyle{amsplain}
\bibliography{$HOME/Academie/Recherche/_latex/bibliotheca}%

\providecommand{\bysame}{\leavevmode\hbox to3em{\hrulefill}\thinspace}
\providecommand{\MR}{\relax\ifhmode\unskip\space\fi MR }
\providecommand{\MRhref}[2]{%
  \href{http://www.ams.org/mathscinet-getitem?mr=#1}{#2}
}
\providecommand{\href}[2]{#2}
\begin{thebibliography}{10}

\bibitem{bingham89}
N.~H. Bingham, C.~M. Goldie, and J.~L. Teugels, \emph{Regular variation},
  Cambridge Univ. Press, 1989.

\bibitem{bopspe89}
R.~B. Boppana and J.~H. Spencer, \emph{A useful elementary correlation
  inequality}, J. Comb. Theory A \textbf{50} (1989), 305--307.

\bibitem{erdo56}
P.~Erd\H{o}s, \emph{Problems and results in additive number theory}, Colloque
  sur la Theorie des Nombres (CBRM) (Bruxelles), 1956, pp.~127--137.

\bibitem{erdtet90}
P.~Erd\H{o}s and P.~Tetali, \emph{Representations of integers as the sum of $k$
  terms}, Random Structures Algorithms \textbf{1} (1990), 245--261.

\bibitem{fang24}
J.-H. Fang, \emph{On {C}illeruelo--{N}athanson's method in {S}idon sets}, J.
  Number Theory \textbf{260} (2024), 223--231.

\bibitem{halberstam83}
H.~Halberstam and K.~F. Roth, \emph{Sequences}, revised ed., Springer, 1983.

\bibitem{kimvvu00}
J.~H. Kim and V.~H. Vu, \emph{Concentration of multivariate polynomials and its
  applications}, Combinatorica \textbf{20} (2000), no.~3, 417--434.

\bibitem{nat05}
M.~B. Nathanson, \emph{Every function is the representation function of an
  additive basis for the integers}, Port. Math. \textbf{62} (2005), 55--72.

\bibitem{nath10}
\bysame, \emph{Cassels bases}, Additive number theory: {F}estschrift in honor
  of the sixtieth birthday of {M}elvyn {B}. {N}athanson (D.~Chudnovsky and
  G.~Chudnovsky, eds.), Springer, New York, 2010, pp.~259--285.

\bibitem{taf19}
C.~T\'{a}fula, \emph{An extension of the {Erd\H{o}s--Tetali} theorem}, Random
  Structures Algorithms \textbf{55} (2019), no.~1, 173--214.

\bibitem{taovu06}
T.~Tao and V.~H. Vu, \emph{Additive combinatorics}, Cambridge Stud. Adv. Math.,
  vol. 105, Cambridge Univ. Press, 2006.

\bibitem{vvu00wp}
V.~H. Vu, \emph{On a refinement of {W}aring's problem}, Duke Math. J.
  \textbf{105} (2000), 107--134.

\bibitem{vvu00mc}
\bysame, \emph{On the concentration of multivariate polynomials with small
  expectation}, Random Structures Algorithms \textbf{16} (2000), 344--363.

\bibitem{yan06}
J.-A. Yan, \emph{A simple proof of two generalized {B}orel--{C}antelli lemmas},
  In memoriam {Paul-Andr\'e Meyer} - S\'eminaire de Probabilit\'es XXXIX
  (M.~\'Emery and M.~Yor, eds.), Lecture Notes in Mathematics, vol. 1874,
  Springer Berlin, Heidelberg, 2006, pp.~77--79.

\end{thebibliography}
\end{document}